\documentclass[a4paper,11pt]{amsart}
\usepackage{ amsbsy, amscd, amsfonts, amsmath, amssymb, amsthm, amsxtra, bbm, braket, comment, fancyhdr, graphics, latexsym, mathabx, mathrsfs, mathtools, verbatim} 

\usepackage{a4wide} 

\usepackage[usenames,dvipsnames]{xcolor}

\usepackage[shortlabels]{enumitem}
\SetEnumerateShortLabel{a}{\textup{(\alph*)}}
\SetEnumerateShortLabel{A}{\textup{(\Alph*)}}
\SetEnumerateShortLabel{1}{\textup{(\arabic*)}}
\SetEnumerateShortLabel{i}{\textup{(\roman*)}}
\SetEnumerateShortLabel{I}{\textup{(\Roman*)}}

\usepackage[mathscr]{euscript}

\usepackage[pdfborder={0 0 0}]{hyperref} 
\hypersetup{breaklinks, raiselinks}

\usepackage[initials,lite,alphabetic]{amsrefs} 

\theoremstyle{plain}
\newtheorem{thm}{Theorem}[section]

\newtheorem{theorem}[thm]{Theorem}

\newtheorem{lemma}[thm]{Lemma}

\newtheorem*{proper*}{Property}
\newtheorem{proposition}[thm]{Proposition}

\newtheorem{corollary}[thm]{Corollary}

\newtheorem*{lm*}{Lemma}
\newtheorem*{thm*}{Theorem}

\theoremstyle{definition}

\newtheorem{definition}[thm]{Definition}

\newtheorem*{df*}{Definition}

\newtheorem{dis}[thm]{Discussion}

\newtheorem{example}[thm]{Example}

\newtheorem{ex-notn}[thm]{Example/Notation}

\newtheorem{construction}[thm]{Construction}

\theoremstyle{remark}

\newtheorem{observation}[thm]{Observation}

\newtheorem{remark}[thm]{Remark}

\newtheorem*{acknowledgement*}{Acknowledgement}

\newtheorem*{ex*}{Example}
\newtheorem*{exer*}{Exercise}
\newtheorem*{rem*}{Remark}
\newtheorem*{prob*}{Problem}
\newtheorem*{prop*}{Proposition}

\def\link{\operatorname{link}}

\def\dN{\operatorname{dN}}

\def\uN{\operatorname{uN}}
\def\ud{\operatorname{ud}}

\def\frakC{\mathfrak{C}}

\def\KK{{\mathbb K}}

\def\NN{{\mathbb N}}

\def\RR{{\mathbb R}}

\def\ZZ{{\mathbb Z}}

\def\calA{\mathcal{A}}

\def\calC{\mathcal{C}}
\def\calD{\mathcal{D}}
\def\calE{\mathcal{E}}
\def\calF{\mathcal{F}}

\def\calH{\mathcal{H}}
\def\calI{\mathcal{I}}

\def\calK{\mathcal{K}}

\def\calV{\mathcal{V}}

\usepackage{bm}

\def\alert#1{\textcolor{Magenta}{#1}}
\def\bar#1{\overline{#1}}

\def\dis{\operatorname{dis}}
\def\deg{\operatorname{deg}}

\def\Index#1{\emph{#1}}

\usepackage{floatrow}
\usepackage{caption}
\DeclareCaptionSubType[alph]{figure}
\AtBeginDocument{%
   \def\MR#1{}
}

\begin{document}
\title{Edgewise strongly shellable clutters}

\author{Jin Guo}
\address{College of Information Science and Technology, Hainan University, Haikou, Hainan, 570228, P.R.~China}
\email{guojinecho@163.com}

\author{Yi-Huang Shen}
\address{Wu Wen-Tsun Key Laboratory of Mathematics of CAS and School of Mathematical Sciences, University of Science and Technology of China, Hefei, Anhui, 230026, P.R.~China}
\email{yhshen@ustc.edu.cn}

\author{Tongsuo Wu}
\address{Department of Mathematics, Shanghai Jiaotong University, Minhang, Shanghai, 200240, P.R.~China}
\email{tswu@sjtu.edu.cn}


\begin{abstract}
When $\calC$ is a chordal clutter in the sense of Woodroofe or Emtander, we show that the complement clutter is edgewise strongly shellable. When $\calC$ is indeed a finite simple graph, we study various characterizations of chordal graphs from the point of view of strong shellability. In particular, the generic graph $G_T$ of a tree is shown to be bi-strongly shellable. We also characterize edgewise strongly shellable bipartite graphs in terms of constructions from upward sequences.
\end{abstract}

\thanks{This research is supported by the National Natural Science Foundation of China (Grant No.~11201445, 11526065 and 11271250) and  the Fundamental Research Funds for the Central Universities. }



\maketitle

\section{Introduction }

Recall that a \Index{simplicial complex} $\Delta$ on the vertex set $\calV(\Delta)=V$ is a finite subset of $2^{V}$, such that $A \in \Delta$ and $B \subseteq A$ implies $B \in \Delta$.
Typically, for the simplicial complexes considered here, the vertex set is $[n]\coloneqq\Set{1,2,\dots,n}$ for some $n\in \ZZ_+$.
The set $A$ is called a \Index{face} if $A \in \Delta$, and called a \Index{facet} if $A$ is a maximal face with respect to inclusion.
Sets of  facets of $\Delta$ will be denoted by $\calF(\Delta)$.
When $\calF(\Delta)=\Set{F_1,\dots,F_t}$, we write $\Delta=\braket{F_1,\dots,F_t}$.
Any set in $2^V\setminus \Delta$ is called a \Index{nonface} of $\Delta$.  The \Index{dimension} of a face $A$, denoted by $\dim(A)$, is $|A|-1$. The \Index{dimension} of a simplicial complex $\Delta$, denoted by $\dim(\Delta)$, is the maximum dimension of its faces. The simplicial complex $\Delta$ is \Index{pure} if all the facets of $\Delta$ have the same dimension.

Recall that a simplicial complex $\Delta$ is called \Index{shellable} if  there exists a linear order $F_1, \dots, F_m$ on its facet set $\mathcal{F}(\Delta)$  such that for each pair $i<j$, there exists a $k<j$, such that $F_j \setminus F_k = \{v\}$ for some $v \in F_j \setminus F_i$.
Such a linear order is called a \Index{shelling order}.
Shellability is an important property when investigating a simplicial complex.  It is well-known that a pure shellable simplicial complex is Cohen-Macaulay, see \cite{MR2724673}.
Matroid complexes, shifted complexes and vertex decomposable complexes are all known to be shellable.

In \cite{SSC}, a stronger requirement was imposed on the shelling orders of shellable simplicial complexes. A simplicial complex $\Delta$ is called \Index{strongly shellable} if its facets can be arranged in a linear order $F_1,\dots,F_t$ in such a way that for each pair $i<j$, there exists $k<j$, such that $|F_j\setminus F_k|=1$ and $F_i\cap F_j\subseteq F_k \subseteq F_i\cup F_j$. Such an ordering of facets will be called a \Index{strong shelling order}. Matroid complexes and pure shifted complexes are all known to be strongly shellable.

As a simple example, let $L_n$ be the line graph with $n$ vertices. Considered as a one-dimensional simplicial complex, $L_n$ is shellable (actually, even vertex decomposable) for all $n$. But $L_n$ is strongly shellable only for $n\le 4$. Additionally, it is well-known that the Stanley-Reisner ideals of the Alexander dual of shellable complexes have linear quotients. An important fact about pure strongly shellable complexes is that the corresponding facet ideals also have linear quotients. Therefore, pure vertex decomposable complexes, not having this property, are in general not strongly shellable. Conversely, we also showed an example in \cite{SSC} that pure strongly shellable complexes are not necessarily vertex decomposable.  Some of the other pertinent facts of strong shellability are summarized in Section 2.

To provide more concrete examples for pure strongly shellable complexes, we show in this paper that strong shellability occurs naturally when considering chordal clutters and graphs. 
Recall that a finite simple graph $G$ is \Index{chordal} if each cycle in $G$ of length at least
$4$ has at least one chord. For finite simple graphs, chordality is an important and fascinating topic. Chordal graphs have many seemingly quite different characterizations, which were later generalized from various perspectives to clutters. 

Note that the facet set $\calF(\Delta)$ of a simplicial complex $\Delta$ gives rise to a clutter $\calC$ whose edge set is $\calF(\Delta)$. When $\Delta$ is strongly shellable, we will call $\calC$ \Index{edgewise strongly shellable}, abbreviated as ESS. In section 3 of the current paper, we will show, roughly speaking, if $\calC$ is a chordal clutter in the sense of Woodroofe \cite{MR2853065} or in the sense of Emtander \cite{MR2603461}, then the complement clutter is ESS; see our Theorems \ref{W-chordal-ss} and \ref{E-chordal-ss}. Since these two types of chordality are mutually non-comparable, the converses of aforementioned two theorems are not true in general, namely, the complement clutter of ESS clutters are generally not chordal in either sense.

Recently, Bigdeli, Yazdan Pour and Zaare-Nahandi \cite{arXiv:1508.03799} also introduced chordal property for $d$-uniform clutters. This class will be denoted by $\frakC_d$. The edge ideals of the complement clutters of the objects in $\frakC_d$ have linear resolutions over any field, but not necessarily have linear quotients. Meanwhile, we can denote the class of $d$-uniform clutters whose complement clutters are ESS, by CESS. A natural question would be whether CESS is a subclass of $\frakC_d$. An answer for this is pertinent to the two questions in \cite{arXiv:1508.03799}; see our Remark \ref{two-questions}.

In Section 4, we will focus on ESS finite simple graphs. As a consequence of the aforementioned results for chordal clutters, we have a new characterization of chordal graphs in Theorem \ref{ssg and chordal graph}: a graph is chordal if and only if the complement graph is ESS. We will provide in-depth investigation of some of the known characterizations of chordal graphs, all from the perspectives of strong shellability.

Recently, Herzog and Rahimi \cite{arXiv:1508.07119} studied the bi-Cohen-Macaulay (abbreviated as bi-CM) property of finite simple graphs. Recall that a simplicial complex $\Delta$ is called \Index{bi-CM}, if both $\Delta$ and its Alexander dual complex $\Delta^{\vee}$ are CM. A simple graph $G$ is called \Index{bi-CM}, if the Stanley-Reisner complex $\Delta_G$ of the edge ideal $I(G)$ of $G$ is bi-CM.
As a complete classification of all bi-CM graphs seems to be impossible, they gave a classification of all bi-CM graphs up to separation. As they showed in \cite[Theorem 11]{arXiv:1508.07119},
the generic graph $G_T$ of a tree $T$ provides a bi-CM inseparable model. In particular, $G_T$ is bi-CM.

Following the same spirit, we will call a simplicial complex $\Delta$ \Index{bi-strongly shellable} (abbreviated as bi-SS), if both $\Delta$ and its Alexander dual complex $\Delta^{\vee}$ are strongly shellable. Bi-SS graphs can be similarly defined. In the present paper, we will show that the generic graph of a tree is not only bi-CM, but also bi-SS; see Theorem \ref{bi-strongly shellable theorem}.

The final section is devoted to bipartite graphs which have no isolated vertex. As a quick corollary of the characterizations of chordal graphs, a bipartite graph with no isolated vertex is ESS if and only if it is a Ferrers graph. We will reconstruct ESS bipartite graphs by upward sequences in Theorem \ref{ssbg=cbuds}. This new point of view will induce the Ferrers-graph characterization of ESS bipartite graphs which have no isolated vertex.

\section{Strongly shellable complexes}
In this section, we summarize some of the properties related to strongly shellable simplicial complexes, that will be applied in this paper.

Recall that a \Index{matroid complex} $\Delta$ is a simplicial complex whose faces are the independent sets of a matroid. By \cite[Proposition III.3.1]{MR1453579}, this is equivalent to saying that $\Delta$ is a simplicial complex such that for every subset $W\subseteq \calV(\Delta)$, the induced subcomplex $\Delta_W\coloneqq \Set{F\in \Delta|F\subseteq W}$ is pure.

\begin{proposition}
    [{\cite[Proposition 6.3]{SSC}}]
    \label{matroid implies ss}
    Matroid complexes are strongly shellable.
\end{proposition}

Let $\Delta$ be a pure simplicial complex. For arbitrary facets $F$ and $G$, the \Index{distance} between them is defined as $\dis(F,G)\coloneqq |F\setminus G|$, which is certainly $|G\setminus F|$ as well. This function satisfies the usual triangle inequality. Sometimes, we will write it as $\dis_\Delta$ to emphasize the underlying simplicial complex.

\begin{lemma}
[{\cite[Lemma 4.2]{SSC}}]
    \label{d=d-1+1}
    Let $\Delta$ be a pure simplicial complex. Then $\Delta$ is  strongly shellable if and only if there exists a linear order $\succ$ on $\mathcal{F}(\Delta)$, such that whenever $F_i\succ F_j$, there exists a facet $F_k\succ F_j$, such that $\dis(F_k, F_j)=1$ and $\dis(F_i, F_k) = \dis(F_i, F_j)-1$.
\end{lemma}

For a pure simplicial complex $\Delta$, its \Index{complement complex} $\Delta^c$ has the facet set $\calF(\Delta^c)=\Set{F^c: F\in \calF(\Delta)}$, where $F^c\coloneqq\calV(\Delta)\setminus F$. The strong shellabilities of $\Delta$ and $\Delta^{c}$ have the following relation:

\begin{lemma}
    [{\cite[Lemma 4.10]{SSC}}]
    \label{complement shellable}
    A pure simplicial complex $\Delta$ is strongly shellable if and only if its complement complex $\Delta^{c}$ has the same property.
\end{lemma}

Let $S=\KK[x_1,\dots,x_n]$ be a polynomial ring over a field $\KK$ and $I$ a graded proper ideal.  Recall that $I$ has \Index{linear quotients}, if there exists a system of homogeneous generators $f_1 ,f_2 ,\dots,f_m$ of $I$ such that the colon ideal $\braket{f_1 ,\dots,f_{i-1}} : f_i$ is generated by linear forms for all $i$. If $I$ has linear quotients, then $I$ is componentwise linear; see \cite[Theorem 8.2.15]{MR2724673}. In particular, if $I$ has linear quotients and can be generated by forms of degree $d$, then it has a $d$-linear resolution; see \cite[Proposition 8.2.1]{MR2724673}.
Another important result in \cite{SSC} is that:

\begin{theorem}
    [{\cite[Theorem 4.11]{SSC}}]
    \label{linear quotients}
    If $\Delta$ is a pure strongly shellable complex, then the facet ideal $I(\Delta)$ has linear quotients. In particular, $I(\Delta)$ has a linear resolution over any field.
\end{theorem}

Recently, Moradi and Khosh-Ahang \cite{arXiv:1601.00456} considered the expansion of simplicial complexes.
    Let $\Delta$ be a simplicial complex with the vertex set $\calV(\Delta)=\Set{x_1,\dots,x_n}$ and $s_1,\dots,s_n$ be arbitrary positive integers.  The \Index{$(s_1,\dots,s_n)$-expansion} of $\Delta$, denoted by $\Delta^{(s_1,\dots,s_n)}$, is the simplicial complex with the vertex set $\Set{x_{i,j}\mid 1\le i\le n, 1\le j\le s_i}$ and the facet set
    \[
    \Set{ \{x_{i_1,r_1},\dots,x_{i_{t},r_{t}}\} \mid \{x_{i_1},\dots,x_{i_t}\}\in \mathcal{F}(\Delta) \text{ and } (r_1,\dots,r_{t})\in [s_{i_1}]\times \cdots \times [s_{i_{t}}]}.
    \]

Now, we wrap up this section with the following strongly shellable version of \cite[Corollary 2.15]{arXiv:1511.04676}.

\begin{theorem}
    [{\cite[Theorem 2.18]{SSC}}]
    \label{expansion-complex}
    Assume that $s_1, \dots, s_n$ are positive integers. Then $\Delta$ is strongly shellable if and only if  $\Delta^{(s_1,\dots,s_n)}$ is so.
\end{theorem}

\section{Edgewise strongly shellable clutters}
Recall that a \Index{clutter} $\calC$  with a finite \Index{vertex set} $\calV(\calC)$ consists of a collection $\calE(\calC)$ of nonempty subsets of $\calV(\calC)$, called \Index{edges}, none of which is included in another. Clutters are also known as \Index{finite simple hypergraphs}.
A \Index{subclutter} $\calK$ of $\calC$ is a clutter such that $\calV(\calK)\subseteq \calV(\calC)$ and $\calE(\calK)\subseteq \calE(\calC)$.  If $U\subseteq \calV(\calC)$, the \Index{induced subclutter} on $U$, $\calC_U$, is the subclutter with $\calV(\calC_U)=U$ and with $\calE(\calC_U)$ consisting of all edges of $\calC$ that lie entirely in $U$.
A clutter $\calC$ is called \Index{$d$-uniform} if $\left|e\right|=d$ for each $e\in \calE(\calC)$.

If $\calC$ is a clutter, we can remove a vertex $v$ in the following two ways.
\begin{enumerate}[a]
    \item The \Index{deletion} $\calC\setminus v$ is the clutter with $\calV(\calC\setminus v)=\calV(\calC)\setminus \Set{v}$ and with
        $$\calE(\calC\setminus v)=\Set{e\in \calE(\calC): v\notin e}.$$
    \item The \Index{contraction} $\calC/v$ is the clutter with $\calV(\calC/v)=\calV(\calC)\setminus \Set{v}$ and with edges the minimal sets of $\Set{e\setminus \Set{v}: e\in \calE(\calC)}$ with respect to inclusion.
\end{enumerate}
Thus, \Index{induced subclutters} are obtained by repeated deletions.
A clutter $\calD$ obtained from $\calC$ by repeated deletions and/or contractions is called a \Index{minor} of $\calC$.

Let $d\le n$ be two positive integers and $V$ a set of cardinality $n$. A subset $S\subseteq \binom{V}{d}$ is called \Index{strongly shellable} if the unique complex $\Delta$ over $V$ whose facet set is $S$, is strongly shellable.
If $\calC$ is a uniform clutter such that the edge set $\calE(\calC)$ is strongly shellable, we say $\calC$ is \Index{edgewise strongly shellable}, abbreviated as ESS.

For a clutter $\calC$, the \Index{vertex-complement clutter} $\calC^{vc}$ is the clutter on $\calV(\calC)$ such that
\[
\calE(\calC^{vc})=\Set{\calV(\calC)\setminus E: E\in \calE(\calC)}.
\]
On the other hand, for a $d$-uniform clutter $\calC$, the \Index{edge-complement clutter} $\calC^{ec}$ is the clutter on $\calV(\calC)$ such that
\[
\calE(\calV^{ec})=\{E: \left|E\right|=d \text{ and } E\notin \calE(\calC)\}.
\]

The following lemma is a re-statement of Lemma \ref{complement shellable}:

\begin{lemma}
    \label{vc}
    Let $\calC$ be a uniform clutter. Then $\calC$ is ESS if and only if its vertex-complement clutter $\calC^{vc}$ is ESS.
\end{lemma}

If $\calC$ is a clutter, an \Index{independent set} of $\calC$ is a subset of $\calV(\calC)$ containing no edge of $\calC$. The \Index{independence complex} is
\[
\calI(\calC)\coloneqq\Set{F\subseteq \calV(\calC): \text{ $F$ is an independent set of $\calC$}}.
\]
Fix an integer $d$, which will usually be the minimum edge cardinality of $\calC$. Let $c_d(\calC)$ be the clutter such that $\calV(c_d(\calC))=\calV(\calC)$ and
\[
\calE(c_d(\calC))=\{e\subset \calV(\calC): |e|=d \text{ and }e \notin \calE(\calC)\}.
\]
Edges of $c_d(\calC)$ will be referred to as \Index{$d$-non-edges} of $\calC$.
In the special case that $\calC$ is $d$-uniform, $c_d(\calC)$ is exactly the edge-complement clutter of $\calC$.
In general, when $\calC$ is not necessarily uniform, it is not difficult to see that $\calI(c_d(\calC))^\vee$ is the pure simplicial complex whose facet set is
\[
\Set{\calV(\calC)\setminus e: e\in c_d(\calC)}.
\]
Hence by Lemma \ref{vc}, $\calI(c_d(\calC))^\vee$ is strongly shellable if and only if $c_d(\calC)$ is ESS. One also observes that if $\calC$ is $d$-uniform, then $\calF(\calI(c_d(\calC))^\vee)=\calE((\calC^{ec})^{vc})$.

In this section, we will deal with the strong shellability of $\calI(c_d(\calC))^\vee$ when $\calC$ is a chordal clutter. The chordality of finite simple graphs is an important topic in graph theory and commutative algebra. There are many different characterizations; see, for instance, the discussions in \cite{MR0130190, MR2724673}. There are also many non-equivalent generalizations of chordality to higher dimensions. For instance, in the following, we will study the two different chordalities introduced by Woodroofe in \cite{MR2853065} and by Emtander in \cite{MR2603461} respectively. For their relationship, please refer to \cite[Example 4.8]{MR2853065}.

\subsection{Woodroofe's chordality}
Let $\calC$ be a clutter. A vertex $v$ of $\calC$ is called \Index{simplicial} if for every two distinct edges $e_1$ and $e_2$ of $\calC$ that contain $v$, there is a third edge $e_3$ such that $e_3\subseteq (e_1\cup e_2)\setminus \Set{v}$.
A clutter $\calC$ is called \Index{W-chordal} if every minor of $\calC$ has a simplicial vertex.

Recall that
given a simplicial complex $\Delta$ and a vertex $v$, the \Index{deletion} complex $\Delta\setminus v$  is the simplicial complex
\[
\Delta\setminus v\coloneqq\Set{F\in\Delta\mid v\not \in F},
\]
and the \Index{link} complex $\link_\Delta(v)$ is the simplicial complex
\[
\link_\Delta(v)\coloneqq\Set{F\in \Delta\mid v\notin F \text{ and }\{v\}\cup F \in \Delta}.
\]
On the other hand, a vertex $v$ of a simplicial complex $\Delta$ is called a \Index{shedding vertex} if no facet of $\link_\Delta v$ is a facet of $\Delta\setminus v$.

\begin{theorem}
    \label{W-chordal-ss}
    If $\calC$ is a W-chordal clutter with minimum edge cardinality $d$, then $\calI(c_d(\calC))^\vee$ is strongly shellable.
\end{theorem}

\begin{proof}
    We proceed by induction, with base cases as follows.
    \begin{enumerate}[i]
        \item If $c_d(\calC)$ has no edge, then $\calI(c_d(\calC))^\vee$ is the degenerate complex $\{\}$. There is nothing to show here.
        \item If $d=1$ and there is an edge in $c_d(\calC)$, the clutter $c_d(\calC)$ is trivially ESS. Therefore, $\calI(c_d(\calC))^\vee$ is strongly shellable.
    \end{enumerate}

    For $d\ge 2$, let $v$ be a simplicial vertex of $\calC$. By the proof of \cite[Theorem 6.9]{MR2853065}, we have the following two key observations.
    \begin{enumerate}[a]
        \item \label{part-a} The link
            \[
            \link_{\calI(c_d(\calC))^\vee} (v)=\calI(c_d(\calC)\setminus v)^\vee = \calI(c_d(\calC\setminus v))^\vee.
            \]
            As $\calC\setminus v$ is W-chordal by definition, and  has minimum edge cardinality at least $d$, the complex $\calI(c_d(\calC\setminus v))^\vee$ is either degenerate or by induction strongly shellable.
        \item \label{part-b} The deletion
            \[
            \calI(c_d(\calC))^\vee\setminus v= \calI(c_d(\calC)/v)^\vee=\calI(c_{d-1}(\calC/v))^\vee. \label{eqn:deletion}
            \]
            As $\calC/v$ is W-chordal by definition, and has minimum edge cardinality at least $d-1$, the complex $\calI(c_{d-1}(\calC/v))^\vee$ is either degenerate or by induction strongly shellable.
    \end{enumerate}

    If $\calC\setminus v$ has no non-edge of cardinality $d$, then the vertex $v$ is indeed contained in every edge of $c_d(\calC)$, hence in no facet of $\calI(c_d(\calC))^\vee$. In this case, $\calI(c_d(\calC))^\vee=\calI(c_d(\calC))^\vee \setminus v$, which by part \ref{part-b} is strongly shellable.

    Otherwise, $\calC\setminus v$ has at least one non-edge of cardinality $d$. Then $v$ is a shedding vertex in $\calI(c_d(\calC))^\vee$ by \cite[Lemma 6.8]{MR2853065}. Let $F_1,F_2,\dots,F_m$ be a strong shelling order of the facets of $\link_{\calI(c_d(\calC))^\vee} (v)$ in part \ref{part-a}. Likewise, let $G_1,G_2,\dots,G_n$ be a strong shelling order of the facets of  $\calI(c_d(\calC))^\vee\setminus v$  in part \ref{part-b}. As $v$ is a shedding vertex,  the facet set of $\calI(c_d(\calC))^\vee$ consists of
    \[
    G_1,\dots,G_n, F_1\cup \Set{v},\dots,F_m\cup\Set{v}.
    \]
    We claim that this is a strong shelling order.

    It suffice to show that for arbitrary $i\in[n]$ and $j\in[m]$, we can find suitable $i'\in[n]$ such that
    \[
    \dis(G_{i'},F_j\cup \Set{v})=1\text{ and }\dis(G_{i'},G_i)=\dis(G_i,F_j\cup\Set{v})-1.
    \]
    Equivalently, we show that
    \[
    \dis(G_{i'}^c,F_j^c \setminus v)=1 \text{ and } \dis(G_{i'}^c,G_{i}^c)=\dis(G_i^c,F_j^c\setminus v)-1.
    \]
    The complements are taken with respect to the vertex set $\calV(\calC)$. Therefore, these sets are edges in $c_d(\calC)$. The case when $\dis(G_i^c,F_j^c\setminus v)=1$ is trivial. Thus, we may assume that $\dis(G_i^c,F_j^c\setminus v)\ge 2$. Note that both $G_i^c$ and $F_j^c$ contain $v$. We only need to find suitable $u\in  (F_j^c\setminus v)\setminus G_i^c$ such that $F_j^c \setminus u\in c_d(\calC)$.  Whence, we can take $G_{i'}$ with $G_{i'}^c=F_j^c \setminus u$.

    If we cannot find such a vertex $u$, since $\dis(G_i^c,F_j^c\setminus v)\ge 2$, we can take two distinct $u_1,u_2\in (F_j^c\setminus v)\setminus G_i^c$ such that both $e_1\coloneqq F_j^c\setminus u_1$ and $e_2\coloneqq F_j^c\setminus {u_2}$ belong to $\calC$. As $v$ is simplicial for $\calC$ and belongs to both $e_1$ and $e_2$, we can find suitable edge $e\subseteq (e_1\cup e_2)\setminus v=F_j^c\setminus v$. But the cardinality of $F_j^c \setminus v$ is exactly $d$, the minimum edge cardinality of $\calC$. Hence $e=F_j^c\setminus v$ is an edge of $\calC$ not containing $v$. This contradicts to the assumption that $F_j^c\setminus v \in c_d(\calC)$.
\end{proof}

\subsection{Emtander's chordality}
Two distinct vertices $x,y$ of a clutter $\calH$ are \Index{neighbors} if there is an edge $E\in \calE(\calH)$, such that $x,y\in E$. For any vertex $x\in \calV(\calH)$, the  set of neighbors of $x$ is denoted by $N(x)$, and called the \Index{neighborhood} of $x$. If $N(x)=\varnothing$, $x$ is called \Index{isolated}. Furthermore, we let $N[x]\coloneqq N(x)\cup \Set{x}$ be the \Index{closed neighborhood} of $x$.

The \Index{$d$-complete clutter}, $K_n^d$, on $[n]$, is defined by  $\calE(K_n^d)=\binom{[n]}{d}$, the set of all subsets of $[n]$ of cardinality $d$. If $n<d$, $K_n^d$ is interpreted as $n$ isolated points.

A clutter $\calC$ is said to have a \Index{perfect elimination order} if its vertices can be ordered $x_1,x_2,\dots, x_n$ such that for each $i$, either the induced subclutter $\calC_{N[x_i]\cap \{x_i,x_{i+1},\dots,x_n\}}$ is isomorphic to a $d$-complete clutter $K_m^d$ for some $m\ge d$, or else $x_i$ is isolated in $\calC_{\{x_i,x_{i+1},\dots,x_n\}}$. By \cite[Lemma 2.2]{MR2603461}, if $\calC$ has a perfect elimination order and $\calE(\calC)\ne \emptyset$, then $\calC$ has a perfect elimination order $x_1,\dots,x_n$ such that $x_1$ is not isolated.

An \Index{E-chordal clutter} is a $d$-uniform clutter, obtained inductively as follows:
\begin{enumerate}[a]
    \item $K_n^d$ is an E-chordal clutter for any $n,d\in \NN$.
    \item If $\calK$ is E-chordal, then so is $\calC=\calK\cup_{K_j^d}K_i^d$ for some $0\le j<i$. Here, we attach $K_i^d$ to $\calK$ by identifying suitable common $K_j^d$.
\end{enumerate}

\begin{lemma}
    [{\cite[Theorem-Definition 2.1]{MR2603461}}]
    \label{perfect-elimination}
    Let $\calC$ be a $d$-uniform clutter. Then $\calC$ is E-chordal if and only if it has a perfect elimination order.
\end{lemma}

\begin{lemma}
    [{\cite[Theorem 4.3]{MR2853077}}]
    Let $\calC$ be a $d$-uniform E-chordal clutter. Then $ \calI(c_d(\calC))^\vee$ is shellable.
\end{lemma}

The above shellability result can be strengthened as follows.

\begin{theorem}
    \label{E-chordal-ss}
    Let $\calC$ be a $d$-uniform E-chordal clutter.
    Then $ \calI(c_d(\calC))^\vee$ is strongly shellable.
\end{theorem}

\begin{lemma}
    \label{layers of matroids}
    Let $X$ and $Y$ be two disjoint finite sets. Let
    \[
    B_{p,q}=\{X'\cup Y':X'\subseteq X \text{ and } Y'\subseteq Y \text{ with }|X'|=p \text{ and }|Y'|=q\}.
    \]
    Fix a positive integer $\lambda$. For indices $i$ and $j$ with $\max(0,\lambda-|Y|)\le i\le j \le \min(\lambda,|X|)$, let
    \[
    A_{i,j}=\bigcup_{i\le k\le j}B_{k,\lambda-k}.
    \]
    Then $A_{i,j}$ is strongly shellable.
\end{lemma}

\begin{proof}
    Let $\Delta$ be the simplicial complex whose facet set is $A_{i,j}$.
    Obviously,
    \[
    \Delta=\{X'\cup Y' :   \text{ $X'\subseteq X$ and $Y'\subseteq Y$ with $|X'|\le j$, $|Y'|\le \lambda-i$ and $|X'|+|Y'|\le \lambda$} \}.
    \]
    For every subset $W\subseteq X\cup Y$, the induced subcomplex
    \begin{align*}
        \Delta_W= \{X'\cup Y' :&   \text{ $X'\subseteq X\cap W$ and $Y'\subseteq Y\cap W$ with}\\
        & \text{ $|X'|\le j$, $|Y'|\le \lambda-i$ and $|X'|+|Y'|\le \lambda$} \}.
    \end{align*}
    This subcomplex is clearly pure. Hence $\Delta$ is a matroid complex. By Proposition \ref{matroid implies ss}, $\Delta$ is strongly shellable, i.e., $A_{i,j}$ is strongly shellable.
\end{proof}

\begin{proof}
    [Proof of Theorem \ref{E-chordal-ss}]
    We prove by induction. The base cases have already been discussed in the proof for Theorem \ref{W-chordal-ss}. Since $\calC$ is E-chordal, it has a perfect elimination order $x_1,x_2,\dots,x_n$ on the vertex set. Take $v=x_1$.
    Note that $\calC_{\{x_2,\dots,x_n\}}$ is still E-chordal. Thus, by induction, the link complex $\link_{\calI(c_d(\calC))^\vee}(v)=\calI(c_d(\calC\setminus v))^\vee$ is strongly shellable. On the other hand, we always have $\calI(c_d(\calC))^\vee\setminus v= \calI(c_d(\calC)/v)^\vee$.

    By the elimination condition, the vertex $v$ is clearly a simplicial vertex.  As in the proof for Theorem \ref{W-chordal-ss}, either $v$ is a shedding vertex, or else $\calI(c_d(\calC))^\vee=\calI(c_d(\calC))^\vee\setminus v$. In both cases, we are reduced to consider $\calI(c_d(\calC)/v)^\vee$.

    Let $N=N(x_1)$ be the neighborhood of $x_1$. By \cite[Lemma 6.7]{MR2853065}, $c_d(\calC)/v=c_{d-1}(\calC/v)$. Its edge set is
    \[
    \big \{e\subseteq \{x_2,\dots,x_n\}: |e|=d-1 \text{ and }e\not\subset N \big\}.
    \]
    If we take $X=\Set{x_2,\dots,x_n}\setminus N$ and $Y=N$, then the above edge set is $A_{1,\min(d-1,n-1-|N|)}$ for $\lambda=d-1$ in Lemma \ref{layers of matroids}. Thus, it is strongly shellable. Equivalently, $\calI(c_d(\calC)/v)^\vee$ is strongly shellable.

    The rest of the proof will be essentially the same as that for Theorem \ref{W-chordal-ss}.
\end{proof}

\begin{remark}
    Under the same assumptions as in Theorems \ref{W-chordal-ss} and \ref{E-chordal-ss}, Woodroofe \cite[Theorem 6.9, Proposition 6.11]{MR2853065} showed that $\calI(c_d(\calC))^\vee$ is also vertex decomposable respectively. On the other hand, by \cite[Examples 6.7 and 6.9]{SSC}, we know that there is no implication between strongly shellable complexes and vertex decomposable complexes.
\end{remark}

\begin{remark}
    \label{two-questions}
    Bigdeli, Yazdan Pour and Zaare-Nahandi \cite{arXiv:1508.03799} also introduced chordal property for $d$-uniform clutters. Following their notation, denote this class by $\frakC_d$.  Among others, they showed the following relation for $d$-uniform clutters:
    \begin{center}
        \begin{tabular}{r}
            W-chordal\\
            E-chordal
        \end{tabular}
        $\bigg\} \subsetneq  \mathfrak{C}_d \subseteq \mbox{LinRes}$
    \end{center}
    Here, LinRes is the class of $d$-uniform clutters whose edge ideals have a linear resolution over any field. Furthermore, they asked the following two questions.
    \begin{enumerate}[1]
        \item Does there exist any $d$-uniform clutter $\calC$ such that the ideal $I (\bar{\calC})$ has a linear resolution over any field, but $\calC$ is not in the class $\mathfrak{C}_d$?
        \item Find a subclass of chordal clutters $\calC$ such that their associated ideals $I(\bar{\calC})$ have linear quotients.
    \end{enumerate}
    On the other hand, by Lemma \ref{complement shellable}, Theorems \ref{linear quotients}, \ref{W-chordal-ss} and \ref{E-chordal-ss}, we have
    \begin{center}
        \begin{tabular}{r}
            W-chordal\\
            E-chordal
        \end{tabular}
        $\bigg\} \subsetneq  \mbox{CESS} \subseteq \mbox{LinRes}$
    \end{center}
    Here, CESS denotes the class of $d$-uniform clutters whose (edge)-complement clutter is ESS. By this fact, one might be tempted to ask if CESS is a subclass of $\frakC_d$. For the time being, we don't have an answer, though extensive computational examples suggest so. Matroid complexes are typical strongly shellable complexes. But as a matter of fact, it is not clear so far whether the complement of matroid complexes belong to $\frakC_d$; see \cite[Proposition 2.2]{arXiv:1601.03207} and the remark after that. Only one thing is for sure: if CESS is not a subclass of $\frakC_d$, then we have an answer for the first question.
\end{remark}

\section{Edgewise strongly shellable graphs}

Throughout this section, we focus on 2-uniform clutters, which are actually finite simple graphs.  Recall that if $G$ is a \Index{finite simple graph} with the vertex set $\calV(G)$ and the edge set $\calE(G)$,  then $|\calV(G)|$ is finite and $\calE(G)\subseteq \Set{ \{u,v\}\mid \text{$u,v\in \calV(G)$ are distinct}}$.
The \Index{complement graph} $\bar{G}$ is the finite simple graph with identical vertex set and $\{u,v\}\in \calE(\bar{G})$ if and only if $\{u,v\}\notin \calE(G)$ for distinct $u,v\in \calV(G)$.
The \Index{neighborhood} of $v$ in $G$ is $N_G(v)\coloneqq\set{u\in \calV(G)\mid \{u,v\}\in\calE(G)}$ and its cardinality is called the \Index{degree} of $v$.
And a vertex $v\in \calV(G)$ is \Index{isolated} if its degree is $0$, i.e., there exists no edge containing $v$ in $G$. A vertex is a \Index{leaf} if its degree is $1$.

Recall that a \Index{cycle} of $G$ of length $q\ge 3$ is a subgraph $C$ of $G$ such that
\[
\calE(C) = \Set{\Set{i_1 ,i_2 },\Set{i_2 ,i_3 }, \dots ,\Set{i_{q-1} ,i_q },\Set{i_q ,i_1 }},
\]
where $i_1 ,i_2 ,\dots,i_q$ are distinct vertices of $G$. A \Index{chord} of a cycle $C$ is an edge $\Set{i,j}$ of $G$ such that $i$ and $j$ are vertices of $C$ with $\Set{i,j} \notin \calE(C)$. A \Index{chordal graph} is a finite graph each of whose cycles of length $> 3$ has a chord. Note that W-chordal 2-uniform clutters, E-chordal 2-uniform clutters and chordal graphs coincide.

Given a finite simple graph $G$, a subset $C$ of $\calV(G)$ is called a \Index{clique} of $G$ if for all distinct $i$ and $j$ in $C$, one has $\Set{i,j}\in \calE(G)$.  A \Index{perfect elimination ordering} of $G$ is an ordering $i_1,\dots,i_n$ of the vertices of G such that for each $j$ with $1 \le j < n$,
\[
C_{i_j}= \big\{i_k \in \calV(G) \mid   j<k\le n,\{i_k ,i_j \} \in \calE(G)\big\}
\]
is a clique of $G$. This is obviously a special case of the perfect elimination order that we encountered when discussing Emtander's chordality. The following result is clear.

\begin{lemma}
    \label{peo-restriction}
    Suppose that $O: v_1,\dots,v_n$ gives a perfect elimination order for $G$. Then for arbitrary nonempty subset $V'$ of $\calV(G)$, the restriction of $O$ on $V'$ gives a perfect elimination order for the induced subgraph $G_{V'}$.
\end{lemma}

For a finite simple graph $G$ with $\calV(G)=[n]$, its edge ideal is
\[
I(G)=\braket{x_ix_j\mid \{i,j\}\in\calE(G)} \subset S=\KK[x_1,\dots,x_n]
\]
for some base field $\KK$.
If instead $\calV(G)=\Set{x_1,\dots,x_n}$, then the edge ideal is
\[
I(G)=\braket{x_ix_j\mid \{x_i,x_j\}\in \calE(G)}\subset S.
\]

\begin{remark}
    \label{ESS-graph}
    Let $G$ be a finite simple graph. Notice that a strong shelling order $\succ$ on the edge set $\calE(G)=\{E_1,\dots,E_t\}$ simply means that whenever we have two disjoint edges $E_i\succ E_j$, then we can find some $E_k \succ E_j$ that intersects both $E_i$ and $E_j$ non-trivially. In other words, to check the strong shellability of $\succ$, one only needs to check non-adjacent edges.
\end{remark}

The main result of this section is as follows.

\begin{theorem}
    \label{ssg and chordal graph}
    Let $G$ be a finite simple graph. Then the following conditions are equivalent:
    \begin{enumerate}[1]
        \item \label{SSGC-1} $G$ is ESS.
        \item \label{SSGC-2} The edge ideal $I(G)$ has linear quotients.
        \item \label{SSGC-3} The edge ideal $I(G)$ has a linear resolution.
        \item \label{SSGC-4} The complement graph $\bar{G}$ is chordal.
        \item \label{SSGC-5} The complement graph $\bar{G}$ has a perfect elimination ordering.
    \end{enumerate}
\end{theorem}

\begin{proof}
    The ``$\ref{SSGC-1}\Rightarrow \ref{SSGC-2}$'' part is due to Theorem \ref{linear quotients}. The ``$\ref{SSGC-2}\Rightarrow \ref{SSGC-3}$'' part is well-known, cf.~\cite[Lemma 1.5]{MR1918513}. The equivalence between \ref{SSGC-3} and \ref{SSGC-4} is due to \cite{MR1171260}.  And the equivalence between $\ref{SSGC-4}$ and $\ref{SSGC-5}$, as far as we know, can be traced back to \cite{MR0130190,MR0186421,MR0270957}.
    The ``$\ref{SSGC-4}\Rightarrow \ref{SSGC-1}$'' part is due to Theorems \ref{W-chordal-ss} or \ref{E-chordal-ss}, together with Lemma \ref{complement shellable}.
\end{proof}

For completeness, we will also demonstrate the direct proofs of $\ref{SSGC-5}\Rightarrow \ref{SSGC-1}\Rightarrow \ref{SSGC-4}$ and $\ref{SSGC-3}\Rightarrow \ref{SSGC-1}$ respectively. Before that, we first show two quick corollaries to the above characterization.

\begin{definition}
    [{\cite[Definition 2.3]{arXiv:1601.00456}}]
    Let $\calC$ be a clutter with the vertex set $\calV(\calC)=\Set{x_1,\dots,x_n}$ and $s_1,\dots,s_n$ be arbitrary positive integers.  The $(s_1,\dots,s_n)$-expansion of $\calC$, denoted by $\calC^{(s_1,\dots,s_n)}$, is the clutter with the vertex set $\Set{x_{i,j}\mid 1\le i\le n, 1\le j\le s_i}$ and the edge set
    \begin{align*}
        \Set{\{x_{i_1,r_1},\dots, x_{i_t,r_t}\} \mid  \{x_{i_1},\dots, x_{i_t}\}\in \mathcal{E}(\mathcal{C}) \text{ and } (r_1,\dots,r_{t})\in [s_{i_1}]\times \cdots \times [s_{i_{t}}]} \cup \\
        \Set{\{x_{i,j},x_{i,k}\} \mid \ 1\le i\le n, \ 1\le j\neq k\le s_i}.
    \end{align*}
\end{definition}

Now, we can recover the following result of \cite[Proposition 2.6]{arXiv:1601.00456}.

\begin{corollary}
    \label{expansion-complex-chordal}
    For any positive integers $s_1,\ldots,s_n$, $G$ is a chordal graph if and only if $G^{(s_1,\dots,s_n)}$ is chordal.
\end{corollary}

\begin{proof}
    Let $\Delta$ be the one-dimensional pure simplicial complex over $\calV(G)=\Set{x_1,\dots,x_n}$ whose facet set is $\calE(\bar{G})$ where $\bar{G}$ is the complement graph of $G$ with respect to $\calV(G)$. Let $G'$ be the facet graph of $\Delta^{(s_1,\dots,s_n)}$. Obviously, $G'$ is also the complement graph of $G^{(s_1,\dots,s_n)}$ with respect to $\calV(\Delta^{(s_1,\dots,s_n)})=\Set{x_{i,j}\mid 1\le i\le n, 1\le j\le s_i}$. Now, it suffices to apply Theorems \ref{expansion-complex} and \ref{ssg and chordal graph}.
\end{proof}

Recall that for a graph $G$, a property $P$ is \Index{hereditary} if the property $P$ holds for every induced subgraph of $G$ whenever it holds for $G$.

\begin{corollary}
    For finite simple graphs, ESS property is hereditary.
\end{corollary}

\begin{proof}
    It is well-known that chordality is hereditary. Thus, we can apply the equivalence between \ref{SSGC-1} and \ref{SSGC-4} in Theorem \ref{ssg and chordal graph}. 
\end{proof}

Of course, this easy statement can be proved directly and generalized to higher dimensions.

\begin{proposition}
    Let $\calC$ be a uniform ESS clutter and $W\subseteq \calV(\calC)$. Then the induced sub-clutter $\calC_W:=\Set{E\in \calE(\calC):E\subseteq W}$ is also ESS.
\end{proposition}

\begin{proof}
    Suppose that $\succ$ gives a strong shelling order on $\calE(\calC)$ and let $\succ'$ be its restriction on $\calE(\calC_W)$. Take arbitrary distinct $F,G\in \calC_W$ with $F\succ' G$. Then $F\succ G$ with $F\subseteq W$ and $G\subseteq W$. By the strong shellability of $\succ$, one can find suitable $H\in \calE(\calC)$ such that $H\succ G$, $\dis(G,H)=1$ and $F\cap G \subseteq H \subseteq F\cup G$. It follows immediately that $H\subseteq W$, and therefore $H\in \calE(\calC_W)$. By the definition of $\succ'$, we have $H\succ' G$ such that $\dis(G,H)=1$ and $F\cap G \subseteq H \subseteq F\cup G$. This shows that $\succ'$ is a strong shelling order and $\calC_W$ is ESS.
\end{proof}

\subsection{Perfect elimination order, chordality and ESS graphs}
\begin{proof}
    [Proof of $\ref{SSGC-5}\Rightarrow \ref{SSGC-1}$ in Theorem \ref{ssg and chordal graph}]
    We will prove by induction on $|\calV(G)|$ with the base case of $|\calV(G)|=2$ being trivial. Now, suppose that $|\calV(G)|=n\ge 3$ and there exists a perfect elimination order on $\overline{G}$: $v_1, \dots, v_n$. It follows from Lemma \ref{peo-restriction} that $v_{1}, \dots, v_{n-1}$ is also a perfect elimination order for the induced subgraph $\overline{G}_{\{v_1,\dots,v_{n-1}\}}$. By induction, the induced subgraph $G'\coloneqq G_{\{v_1,\dots,v_{n-1}\}}$, as the complement of $\overline{G}_{\{v_1,\dots,v_{n-1}\}}$, is ESS. Therefore, there exists a strong shelling order $\succ'$ on the edge set $\calE(G')$: $e_1, \dots, e_t$. We can define a total order $\succ$ on $\calE(G)$ as follows.
    \begin{enumerate}[a]
        \item If $e_i, e_j \in \calE(G')$, $e_i \succ e_j$ if and only if $e_i \succ' e_j$.
        \item For the edges $\{v_i, v_n\}$ and $\{v_j, v_n\}$, $\{v_i, v_n\} \succ \{v_j, v_n\}$ if and only if $i<j$.
        \item  For each edge $e_i \in \calE(G')$ and edge $\{v_j, v_n\}$, $e_i \succ \{v_j, v_n\}$.
    \end{enumerate}

    In the following, we will show that the order $\succ$ is a strong shelling order on $\calE(G)$. In fact, by Remark \ref{ESS-graph}, we only need to show that for an edge $e_i \in \calE(G')$ and an edge $\{v_j, v_n\}$, if they are not adjacent, then there exists an edge $e \succ \{v_j, v_n\}$ which is adjacent to both $e_i$ and $\{v_j, v_n\}$.

    Assume that $e_i=\{v_l, v_k\}$ and $l<k$. 
    Then $G''\coloneqq \bar{G}_{\{v_l,v_j,v_k,v_n\}}$ has an induced perfect elimination order by Lemma \ref{peo-restriction}.

    Now, consider the following two cases.

    \begin{enumerate}[a]
        \item  If $j < l$, then either $\{v_j, v_l\} \in \calE(G)$ or $\{v_j, v_k\} \in \calE(G)$. In fact, if neither happens, then 
            $\{v_l, v_k\} \in \calE(\overline{G})$ by the definition of perfect elimination ordering. It contradicts to the fact that $\{v_l, v_k\} \in \calE(G)$. Note that $\{v_j, v_l\} \succ \{v_j, v_n\}$ and $\{v_j, v_k\} \succ \{v_j, v_n\}$ in this case.
        \item If $l < j$, then, as the discussion above, either $\{v_l, v_j\} \in \calE(G)$ or $\{v_l, v_n\} \in \calE(G)$. Note that $\{v_l, v_j\} \succ \{v_j, v_n\}$ and $\{v_l, v_n\} \succ \{v_j, v_n\}$ in this case.
    \end{enumerate}

    In both cases, we are able to find an edge $e \succ \{v_j, v_n\}$ which is adjacent to both $\{v_l, v_k\}$ and $\{v_j, v_n\}$. This completes the proof.
\end{proof}

\begin{proof}
    [Proof of $\ref{SSGC-1}\Rightarrow \ref{SSGC-4}$ in Theorem \ref{ssg and chordal graph}]
    Assume that $G$ is ESS. Then, for each pair of edges $\{v_i, v_j\}$ and $\{v_l, v_k\}$ with no common vertex, they are adjacent to a common edge by Remark \ref{ESS-graph}. This implies that one of $\{v_i, v_l\}$, $\{v_i, v_k\}$, $\{v_j, v_l\}$ and $\{v_j, v_k\}$ is an edge in $\calE(G)$. In other words, any cycle with 4 vertices in $\overline{G}$ must have a chord.

    We will complete the proof by contradiction. If $\overline{G}$ is not chordal, then there exists a minimal cycle with more than 3 vertices in $\overline{G}$. Let $e_1, \dots, e_t$ be a strong shelling order on $\calE(G)$.  Let $G_s=G\setminus \Set{e_t,e_{t-1},\dots,e_{s+1}}$. Note that removing edges successively from $\calE(G)$ amounts to adding corresponding edges successively to $\calE(\overline{G})$. Therefore, we can find suitable $s\ge 1$ such that $\overline{G_s}$ contains a minimal cycle with $4$ vertices. Note that $e_1, \dots, e_s$ is a strong shelling order on $\calE(G_s)$. This provides a contradiction.
\end{proof}

\subsection{Linear resolution and ESS graphs}

Given a finite simple graph $G$, a simple graph $H$ is called a \Index{quotient graph} of $G$ if there exists a surjective map $f:\calV(G)\to \calV(H)$, such that
\begin{enumerate}[a]
    \item if $u_1,u_2\in \calV(G)$ such that $\{u_1,u_2\}\in \calE(G)$ and $f(u_1)\ne f(u_2)$, then $\{f(u_1),f(u_2)\}\in \calE(H)$;
    \item if $\{v_1,v_2\}\in \calE(H)$, then there exists $\{u_1,u_2\}\in \calE(G)$ such that $f(u_1)=v_1$ and $f(u_2)=v_2$.
\end{enumerate}
The map $f$ here will be called a \Index{quotient map}.
For simplicity, we will denote the edge $\{f(u_1),f(u_2)\}$ by $f(\{u_1,u_2\})$.

Note that if the edge ideals $I(G)\subset S=\KK[x_g:g\in G]$ and $I(H)\subset T=\KK[x_h:h\in H]$ and let $\varphi_f$ be the natural map from $S$ to $T$ sending $x_g$ with $g\in G$ to $x_{f(g)}$ with $f(g)\in H$, it is possible that $\varphi_f(I(G))\ne I(H)$. This happens if and only if
$\varphi_f(I(G))$ contains square monomials, i.e., for some $h\in H$, one can find some $g_1,g_2\in f^{-1}(h)$ such that $\{g_1,g_2\}\in \calE(G)$. Thus, we will call the quotient map $f$ to be \Index{proper}, if for each $h\in H$, ${f^{-1}(h)}$ is an independent subset of $\calV(G)$. Whence, $H$ is a \Index{proper quotient graph} of $G$.

\begin{example}
    Let $G$ be the simple graph on the left side of Figure \ref{Fig:quotient-graph} over the vertex set $\{a,b_1,b_2,c\}$.  Let $H$ be the simple graph on the right side of Figure \ref{Fig:quotient-graph} over the vertex set $\{a,b,c\}$. The map $f:G\to H$ sending both $b_1$ and $b_2$ to $b$ while keeping other vertices, is a non-proper quotient map. The image of the edge ideal $I(G)$ is $\braket{ab,b^2,bc}$, which is not $I(H)=\braket{ab,bc}$.
    \begin{figure}[!ht]
        \ffigbox[\FBwidth]{}{{
        \begin{subfloatrow}[2]
            \ffigbox{\caption{$G$}}{
            \begin{minipage}[h]{\linewidth} \centering
                \includegraphics{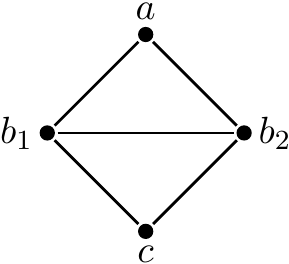}
            \end{minipage}
            }
            \ffigbox{\caption{$H=f(G)$}}{
            \begin{minipage}[h]{\linewidth} \centering
                \includegraphics{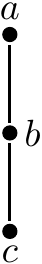}
            \end{minipage}
            }
        \end{subfloatrow}}
        \caption{Quotient graph} \label{Fig:quotient-graph}}
    \end{figure}
\end{example}

\begin{lemma}
    \label{quotient-graph}
    Let $G$ be an ESS graph. If $H$ is a quotient graph of $G$,  then $H$ is also ESS.
\end{lemma}

\begin{proof}
    Let $\succ_G$ be a strong shelling order on $\calE(G)$.  Assume that $f:\calV(G)\to \calV(H)$ is the quotient map.  For each $e\in \calE(H)$, let $\tilde{e}$ be the least edge in $f^{-1}(e)$ with respect to $\succ_G$.  We define the induced order $\succ_H$ as follows:
    \[
    \text{$e_1\succ_H e_2$ if and only if $\tilde{e}_1\succ_G \tilde{e}_2$.}
    \]
    Since $G$ is ESS, when $e_1\succ_H e_2$, either $\tilde{e}_1$ is adjacent to $\tilde{e}_2$, or $\tilde{e}_1$ can be connected with $\tilde{e}_2$ by some edge $\{u_1,u_2\}$ such that $\{u_1,u_2\}\succ_G \tilde{e}_2$. In the former case, $e_1$ is adjacent to $e_2$. In the latter case, if $f(u_1)\ne f(u_2)$, then $e_1$ can be connected with $e_2$ by $e_3=f(\{u_1,u_2\})$ with the property that $e_3\succ_H e_2$. Otherwise, $f(u_1)=f(u_2)$. Hence $e_1$ and $e_2$ are adjacent.
\end{proof}

Let $G$ be a finite simple graph. A \Index{blow-up} of $G$ at the vertex $v\in \calV(G)$ is the new graph $H$ such that
\begin{enumerate}[a]
    \item $\calV(H)=(\calV(G)\setminus{v})\cup \{v_1,\dots,v_m\}$ with $\{v_1,\dots,v_m\}\cap \calV(G)=\varnothing$;
    \item $\calE(H)$ consists of the following two parts:
        \begin{enumerate}[i]
            \item $\calE(G\setminus v)$;
            \item $\big\{ \{v_i,u\} \mid 1\le i \le m \text{ and } \{v,u\}\in \calE(G)\big\}$.
        \end{enumerate}
\end{enumerate}
If graph $G'$ can be obtained from $G$ by a sequence of blow-ups, $G'$ will be called a \Index{blow-up graph} of $G$.

\begin{lemma}
    Let $H$ be a blow-up graph of $G$. If $G$ is ESS, then so is $H$.
    \label{blow-up}
\end{lemma}

\begin{proof}
    Without loss of generality, we may assume that
    $\calV(G)=\Set{x_1,\dots,x_n}$ and
    $H$ is a blow-up of $G$ at the vertex $x_1$. Let $\Delta$ and $\Delta'$ be the one-dimensional pure simplicial complexes whose facet graphs are $G$ and $H$ respectively. Then $\Delta'=\Delta^{(s_1,1,\dots,1)}$ for some positive integer $s_1$. Now, we apply the ``only if'' part of Theorem \ref{expansion-complex}.
\end{proof}

Let $T$ be a tree, i.e., a connected finite simple graph which contains no cycle.  Without loss of generality, we may assume that $\calV(T)=[n]$.  Note that for arbitrary vertices $i$ and $j$ of $T$, there exists a unique path $P:i=i_0,i_1,\dots,i_r=j$ from $i$ to $j$. Following \cite{MR2457194}, $b(i,j)\coloneqq i_1$ is called the \Index{begin} of $P$ and $e(i,j)=i_{r-1}$ is the \Index{end} of $P$. We can attach a \Index{generic matrix} $A(T)$ to $T$ as follows. Assume that $\calE(T)=\Set{e_1,\dots,e_{n-1}}$. For each edge $e_k=\{i,j\}$ with $i<j$, the $k$-th row of $A(T)$ is
\[
r_k=-x_{i,j}\varepsilon_i+x_{j,i}\varepsilon_{j}.
\]
Here $\varepsilon_i$ is the $i$-th canonical unit vector in $\RR^n$.

On the other hand, after \cite{arXiv:1508.07119}, we can associate a special graph $G_T$ to the tree $T$.  The vertices of the graph $G_T$ is given by
\[
\calV(G_T)=\big\{x_{i,j},x_{j,i}\mid \{i,j\}\in \calE(T)\big\}.
\]
And $\{x_{i,k},x_{j,l}\}$ is an edge of $G_T$ if and only if there exists a path $P$ from $i$ to $j$ such that $k=b(i,j)$ and $l=e(i,j)$. The graph $G_T$ will be called the \Index{generic bi-CM graph} (abbreviated as \Index{generic graph}) attached to $T$.

\begin{example}
    \label{example-2-trees}
    For instance, for the two trees in Figure \ref{trees}, we have the corresponding generic matrices
    \[
    A(T_1)=
    \begin{pmatrix}
        -x_{1,2} & x_{2,1} & 0 & 0 \\
        -x_{1,3} & 0 & x_{3,1} & 0 \\
        -x_{1,4} & 0 & 0 & x_{4,1}
    \end{pmatrix}
    \]
    and
    \[
    A(T_2)=
    \begin{pmatrix}
        -x_{1,2} & x_{2,1} & 0 & 0 \\
        0 & -x_{2,3} & x_{3,2} & 0 \\
        0 & 0 & -x_{3,4} & x_{4,3}
    \end{pmatrix},
    \]
    and the attached generic graphs in Figure \ref{GT}.
\end{example}

\begin{figure}[!ht]
    \ffigbox[\FBwidth]{}{{
    \begin{subfloatrow}[2]
        \ffigbox{\caption{$T_1$}}{
            \begin{minipage}[h]{\linewidth} \centering
                \includegraphics{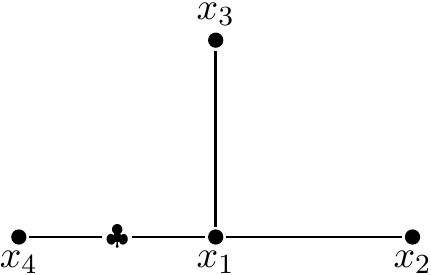}
            \end{minipage}
        }
        \ffigbox{\caption{$T_2$}}{
            \begin{minipage}[h]{\linewidth} \centering
                \includegraphics{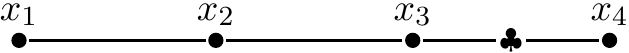}
            \end{minipage}
        }
    \end{subfloatrow}}
    \caption{Some trees} \label{trees} }
\end{figure}

\begin{figure}[!ht]
    \ffigbox[\FBwidth]{}{{
    \begin{subfloatrow}[2]
        \ffigbox{\caption{$G_{T_1}$}}{
            \begin{minipage}[h]{\linewidth} \centering
                \includegraphics{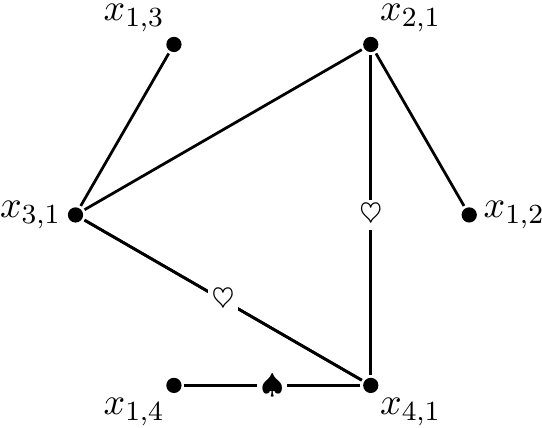}
            \end{minipage}
        }
        \ffigbox{\caption{$G_{T_2}$}}{
            \begin{minipage}[h]{\linewidth} \centering
                \includegraphics{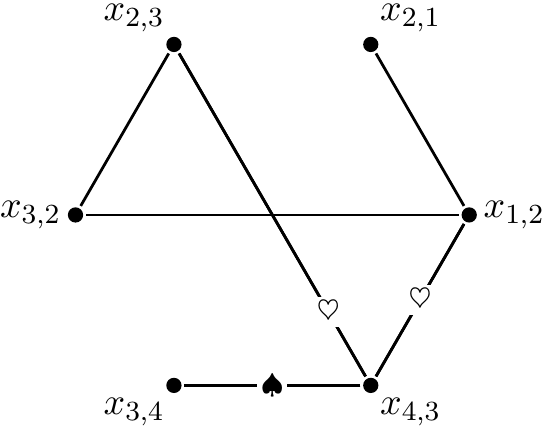}
            \end{minipage}
        }
    \end{subfloatrow}}
    \caption{Corresponding generic graphs} \label{GT} }
\end{figure}

In the following, we will discuss the properties of the generic graph of a tree with at least $3$ vertices, since the $2$-vertices case is clear.   The following observations are easy to check.

\begin{observation}
    \label{obs}
    \begin{enumerate}[a,leftmargin=6mm]
        \item If $i$ is a leaf in $T$, and $k$ is the unique vertex adjacent to $i$, then in the generic graph $G_T$, $x_{k, i}$ is a leaf, and $x_{i, k}$ is the unique vertex adjacent to $x_{k, i}$ in $G_T$.

        \item Let $i, j$ be two leaves in $I$, and let $k, l$ be the unique vertices adjacent to $i$ and $j$ respectively. Then in the unique path connecting $i$ and $j$, $b(i,j)=k$ and $e(i,j)=l$. Hence in the generic graph $G_T$, $\{x_{i,k}, x_{j,l}\} \in \calE(G_T)$.
    \end{enumerate}
\end{observation}

\begin{lemma}
    \label{replaced by leaf}
    Let $G_T$ be the generic graph of a tree $T$. If $x_{i_1,k_1}, x_{i_2,k_2}, \dots, x_{i_t,k_t}$ are adjacent to $x_{j_1,l_1}$ in $G_T$, then there exists a leaf $x_{l_2, j_2}$, such that $x_{j_2,l_2}$ is adjacent to all of $x_{i_1,k_1}, x_{i_2,k_2}, \dots, x_{i_t,k_t}$.
\end{lemma}

\begin{proof}
    If $j_1$ is a leaf in the graph $T$, then $x_{l_1, j_1}$ is a leaf in $G_T$ satisfying the requirement.

    If $j_1$ is not a leaf in the graph $T$, then there exists a leaf $j_2$, such that $b(j_1, j_2) \neq l_1$. Assume that $e(j_1, j_2) = l_2$. It is clear that $x_{l_2, j_2}$ is a leaf in the generic graph $G_T$, and $N(x_{j_1,l_1}) \subseteq N(x_{j_2,l_2})$.
\end{proof}

Given a finite simple graph $G$, the \Index{clique number} of $G$ is the largest cardinality of the cliques in $G$.

\begin{proposition}
    \label{clique number}
    Let $G_T$ be the generic graph of a tree $T$. Then the following quantities coincide:
    \begin{enumerate}[1]
        \item \label{q-1} The clique number of $G_T$;
        \item \label{q-2} the number of leaves in $G_T$;
        \item \label{q-3} the number of leaves in $T$.
    \end{enumerate}
\end{proposition}

\begin{proof}
    Take an arbitrary vertex $x_{l,j}\in \calV(G_T)$.
    If $j$ is a leaf in $T$, then $l$ is the unique vertex adjacent to $j$ in $G$. Thus, by Observation \ref{obs}, $x_{l,j}$ is a leaf in $G_T$.
    On the other hand, if $j$ is not a leaf in $T$, then there exists another vertex, say, $k$, adjacent to $j$ in $G$. As $\{x_{l,j}, x_{j,l}\}, \{x_{l,j}, x_{k,j}\} \in \calE(G_T)$, it is clear that $x_{l,j}$ is not a leaf in $G_T$. Hence, \ref{q-2} is identical to \ref{q-3}.

    Assume that $j_1, j_2, \dots, j_t$ are all the leaves in $T$, and the unique vertices adjacent to them are $l_1, l_2, \dots, l_t$ respectively. We claim that $\Set{x_{j_1,l_1}, x_{j_2,l_2}, \dots, x_{j_t,l_t}}$ is a maximum clique of $G_T$. If fact, for distinct $a,b \in [t]$, it is clear that $b(j_a, j_b)=l_a$ and $e(j_a, j_b)=l_b$. Therefore, $x_{j_1,l_1}, x_{j_2,l_2}, \dots, x_{j_t,l_t}$ is a clique of $G_T$.

    On the other hand, assume that $\Set{x_{i_1, k_1}, x_{i_2, k_2}, \dots, x_{i_s, k_s}}$ is also a maximum clique of $G_T$. By Lemma \ref{replaced by leaf}, if $x_{k_a,i_a}$ is not a leaf of $G_t$, we can replace $x_{i_a, k_a}$ by suitable $x_{j_{a'},l_{a'}}$ and get a clique of same size.
    As we will check all the vertices in this maximum clique and do the replacement whenever necessary, we arrive at the expected inequality $s \leq t$.
    Hence, the maximum clique number is indeed $t$, i.e., \ref{q-1} is identical to \ref{q-2}.
\end{proof}

Next, we study the diameter of the generic graph of a tree.

\begin{lemma}
    \label{there is edge}
    Let $G_T$ be the generic graph of a tree $T$. If $\{x_{i_1,k_1}, x_{j_1,l_1}\}$ and $\{x_{i_2,k_2}, x_{j_2,l_2}\}$ are two edges of $G_T$, then at least one of the following four edges belongs to $\calE(G_T)$: $\{x_{i_1,k_1}, x_{i_2,k_2}\}$, $\{x_{i_1,k_1}, x_{j_2,l_2}\}$, $\{x_{j_1,l_1}, x_{i_2,k_2}\}$, $\{x_{j_1,l_1}, x_{j_2,l_2}\}$.
\end{lemma}

\begin{proof}
    As $\{x_{i_1,k_1}, x_{j_1,l_1}\} \in \calE(G_T)$, there exists a unique path in $T$ connecting $i_1$ and $j_1$, such that $b(i_1,j_1)=k_1$ and $e(i_1,j_1)=l_1$. It is easy to see that for any vertex $a$ in $T$, the farthest vertex among $i_1, k_1, j_1, l_1$ is either $i_1$ or $j_1$. Without loss of generality, we assume that $i_1$ is the farthest one for the vertex $k_2$. We have the following two cases.
    \begin{enumerate}[i]
        \item If the distance $d_T(i_1, i_2)>d_T(i_1, k_2)$, then there exists a path connecting $i_1$ and $i_2$, such that $b(i_1,i_2)=k_1$ and $e(i_1,i_2)=k_2$. Hence $\{x_{i_1,k_1}, x_{i_2,k_2}\} \in \calE(G_T)$.
        \item If $d_T(i_1, i_2)<d_T(i_1, k_2)$, then clearly there exists a path connecting $i_1$ and $j_2$, such that $b(i_1,j_2)=k_1$ and $e(i_1,i_2)=l_2$. Hence $\{x_{i_1,k_1}, x_{j_2,l_2}\} \in \calE(G_T)$. \qedhere
    \end{enumerate}
\end{proof}

\begin{corollary}
    \label{leaf connect with edge}
    Let $i$ be a leaf in $T$, and let $j$ be the unique vertex adjacent to $i$. If $\{x_{k_1,l_1}, x_{k_2,l_2}\}$ is an edge in the generic graph $G_T$, then $\{x_{i,j}, x_{k_1,l_1}\} \in \calE(G_T)$ or $\{x_{i,j}, x_{k_2,l_2}\} \in \calE(G_T)$.
\end{corollary}

\begin{proof}
    Since $\{x_{k_1,l_1}, x_{k_2,l_2}\}, \{x_{i,j}, x_{j,i}\} \in \calE(G_T)$, by Lemma \ref{there is edge}, at least one of  $\{x_{i,j}, x_{k_1,l_1}\},\, \{x_{i,j}$, $x_{k_2,l_2}\}$, $\{x_{j,i}$, $x_{k_1,l_1}\}$ and $\{x_{j,i}, x_{k_2,l_2}\}$ is an edge in $G_T$. Note by Observation \ref{obs} that  $x_{j,i}$ is a leaf in $G_T$, and the unique vertex adjacent to $x_{j,i}$ is $x_{i,j}$. Thus, we have $\{x_{i,j}, x_{k_1,l_1}\} \in \calE(G_T)$ or $\{x_{i,j}, x_{k_2,l_2}\} \in \calE(G_T)$.
\end{proof}

\begin{proposition}
    \label{connect and diameter}
    Let $G_T$ be the generic graph of a tree $T$ with at least $3$ vertices. Then $G_T$ is connected and the diameter of $G_T$ is $3$.
\end{proposition}

\begin{proof}
    Note that for any vertex $x_{i,k}$, there exists at least one edge $\{x_{i,k}, x_{k,i}\}$ adjacent to $x_{i,k}$. So, the connectivity follows from Lemma \ref{there is edge}. Again, by Lemma \ref{there is edge}, the diameter of $G_T$ is at most $3$. Assume that $i, j$ are two leaves, and $\{i, k\}, \{j, l\} \in \calE(T)$.
    Since $T$ has at least $3$ vertices, $k\ne j$ and $i\ne l$.
    It is clear that $\{x_{i,k}, x_{j,l}\} \in \calE(G_T)$, and the distance between $x_{k, i}$ and $x_{l, j}$ in $G_T$ is 3.
\end{proof}

The following is an important property of the generic graph of a tree.

\begin{proposition}
    [{\cite[Proposition 6]{arXiv:1508.07119}}]\label{bi-CM}
    For any tree $T$, the generic graph $G_T$ is bi-CM.
\end{proposition}

We will generalize it and consider bi-strong-shellability; see Theorem \ref{bi-strongly shellable theorem}.


\begin{proposition}
    \label{tree-graph}
    For any tree $T$,  the generic graph $G_T$ is ESS.
\end{proposition}

\begin{proof}
    We prove by induction on $n=|\calV(T)|$. If $n$ is $2$, the situation is clear. Hence we may assume that $n\ge 3$ and for all trees $T'$ with $|\calV(T')|\le n-1$, this proposition holds for $T'$.

    Without loss of generality, we may assume that $\calV(T)=[n]$ and $n$ is a leaf of $T$ which is adjacent to the vertex ${n-1}$ in $T$. Let $T'$ be the tree by deleting $n$ from $T$. Thus the edge set $\calE(G_{T'})$ of the generic graph $G_{T'}$ has a strong shelling order $\succ_{T'}$. We will extend $\succ_{T'}$ to give a strong shelling order on $\calE(G_T)$.
    Note that $G_{T}$ can be built from $G_{T'}$ by attaching the following edges:
    \begin{enumerate}[i]
        \item The ``cone'' part: the edges $\{x_{n,n-1},x_{i,j}\}$ for all $i\in [n-2]$ such that $j=b(i,n)$ for a path from $i$ to $n$ in $T$.
        \item The ``handle'' part: the edge $\{x_{n,n-1},x_{n-1,n}\}$.
    \end{enumerate}

    Let $\succ_{T}$ be an arbitrary total order on $\calE(G_T)$ satisfying:
    \begin{enumerate}[1]
        \item For two edges $e_1, e_2 \in \calE(G_{T'})$, $e_1 \succ_{T} e_2$ if and only if $e_1 \succ_{T'} e_2$;
        \item If $e_1$ is in the ``cone'' part or in the ``handle'' part, and $e_2 \in \calE(G_{T'})$, then $e_1 \succ_{T} e_2$.
    \end{enumerate}

    The existence of such an ordering is without question.
    We claim that $\succ_{T}$ is a strong shelling order on $\calE(G_T)$. In fact, the only case we need to check is when $e_1$ is in the ``cone'' part or in the ``handle'' part, and $e_2 \in \calE(G_{T'})$. Assume that $e_1= \{x_{n,n-1}, x_{i,j}\}$ and $e_2= \{x_{k_1,l_1}, x_{k_2,l_2}\}$. Note that $n$ is a leaf in the graph $T$, by Corollary \ref{leaf connect with edge}. Thus $\{x_{n,n-1}, x_{k_1,l_1}\} \in \calE(G_T)$ or $\{x_{n,n-1} x_{k_2,l_2}\} \in \calE(G_T)$. Note that $\{x_{n,n-1}, x_{k_1,l_1}\} \succ_{T} e_2$ and $\{x_{n,n-1}, x_{k_2,l_2}\} \succ_{T} e_2$ by the construction of $\succ_{T}$. This complete the proof.
\end{proof}

\begin{example}
    To illustrate the previous proof, let's look back at the aforementioned two figures in Example \ref{example-2-trees}.
    In the Figure \ref{trees}, the edge of the leaf that we are considering are labeled with $\clubsuit$. In the Figure \ref{GT}, the ``cone'' part edges are labeled with $\heartsuit$ while the ``handle'' part are labeled with $\spadesuit$.
\end{example}

Recall that if $I$ is a squarefree monomial ideal in $S=\KK[x_1,\dots,x_n]$ generated by squarefree monomials $u_i$, $1\le i\le m$, then the \Index{Alexander dual} of $I$, denoted by $I^\vee$, is defined to be the squarefree monomial ideal
\begin{equation}
    \label{dual}
    I^\vee\coloneqq \bigcap_{i=1}^m \braket{x_j\mid x_j \text{ divides } u_i}. \tag{$\star$}
\end{equation}
It is well-known that $(I^\vee)^\vee=I$. Furthermore, if $I$ is minimally generated by $u_i$, $1\le i \le m$, then above equation \eqref{dual} actually gives a minimal irredudant primary decomposition of $I^{\vee}$.
Finally, we are ready to connect the linear resolution property with the strong shellability in Theorem \ref{ssg and chordal graph}.

\begin{proof}
    [Proof of $\ref{SSGC-3}\Rightarrow \ref{SSGC-1}$ in Theorem \ref{ssg and chordal graph}]
    Without loss of generality, we may assume that $\calV(G)=\Set{x_1,\dots,x_n}$.  Suppose that the quadratice squarefree monomial ideal $I(G)\subset S=\KK[x_1,\dots,x_n]$ has a linear resolution and let $J=I(G)^\vee$ be its Alexander dual ideal. Then, by the well-known Eagon-Reiner theorem \cite[Theorem 3]{MR1633767}, $J$ is Cohen-Macaulay with codimension $2$.
    Suppose that $\Set{u_1,\dots,u_{m+1}}$ is the minimal monomial generating set of $J$.
    Applying the graded Nakayama Lemma to the Taylor resolution of $J$, we have a minimal graded free resolution of of $S/J$ in the form
    \[
    0\to S^m \stackrel{A}{\to} S^{m+1} \to S\to S/J \to 0,
    \]
    since it has to satisfy the Hilbert-Burch theorem \cite[Theorem 1.4.17]{MR1251956}.  Here, $A$ is called a \emph{Hilbert-Burch matrix} of $J$. Each row of $A$ has exactly two nonzero entries, say, on the $i$-th and $j$-th complements with $i<j$, corresponding to a Taylor relation of $u_i$ and $u_j$:
    \[
    r_k=-u_{i,j}\varepsilon_i+u_{j,i}\varepsilon_j,
    \]
    where $u_{i,j}=u_i/\gcd(u_i,u_j)$ and $u_{j,i}=u_j/\gcd(u_i,u_j)$.
    Now we have a finite simple graph $T=T_A$ on $[m+1]$ such that the row of $A$ in above form contributes an edge $\{i,j\}$ to $T$. Since $T$ has $m+1$ vertices, $m$-edges and no isolated vertex, $T$ is indeed a tree. For more details of these preparations, see the discussions in \cite[Remark 6.3]{MR1341789} and \cite{MR2457194}.

    Actually, the matrix $A$ can be obtained from the generic matrix $A(T)$ by the substitution:
    \[
    x_{i,j}\mapsto u_{i,j}.
    \]
    Therefore,
    \[
    J=\bigcap_{i<j} (u_{i,b(i,j)},u_{j,e(i,j)})
    \]
    by \cite[Proposition 1.4]{MR2457194}.

    For each squarefree monomial $u_{i,j}$, we may denote $\deg(u_{i,j})$ by $d(i,j)$ and assume that
    \[
    u_{i,j}=x_{i,j;1}x_{i,j;2}\dots x_{i,j;{d(i,j)}},
    \]
    a product of distinct variables in $S$. We consider the matrix $A_p$ obtained from generic matrix $A(T)$ by the substitution:
    \[
    x_{i,j} \mapsto x_{i,j}^{(1)}x_{i,j}^{(2)}\dots x_{i,j}^{(d(i,j))},
    \]
    a product of distinct new variables.  The maximal minor ideal of $A_p$ is
    \begin{align*}
        J_p= & \bigcap_{i<j} (x_{i,b(i,j)}^{(1)}\dots x_{i,b(i,j)}^{(d(i,b(i,j)))},x_{j,e(i,j)}^{(1)}\dots x_{j,e(i,j)}^{(d(j,e(i,j)))})\\
        = & \bigcap_{i<j} \bigcap_{1\le k_1 \le d(i,b(i,j))} \bigcap_{1\le k_2 \le d(j,e(i,j))} (x_{i,b(i,j)}^{k_1},x_{j,e(i,j)}^{k_2}).
    \end{align*}
    Let $G_p$ be the finite simple graph whose edge ideal $I(G_p)=J_p^\vee$. Obviously $G_p$ is a blow-up graph of the generic graph $G_T$; the vertex $x_{i,j}$ is replaced by the vertices $x_{i,j}^{(k)}$, $1\le k\le d(i,j)$. As $G_T$ is ESS by Proposition \ref{tree-graph}, so is $G_p$ by Lemma \ref{blow-up}.

    On the other hand, $J$ can be obtained from $J_p$ by the substitution:
    \[
    x_{i,j}^{(l)} \mapsto x_{i,j;l}.
    \]
    Notice that $J$ is squarefree with codimension $2$. This implies that $G$ is a proper quotient graph of the graph $G_p$:
    \[
    \calE(G)=\Set{\{x_{i,b(i,j);k_1},x_{j,e(i,j);k_2}\}\mid 1\le i<j\le m+1, 1\le k_1\le d(i,b(i,j)), 1\le k_2\le d(j,e(i,j))}.
    \]
    As $G_p$ is ESS, so is $G$ by Lemma \ref{quotient-graph}.
\end{proof}

The rest of this subsection is devoted to the bi-SS property of the generic graph $G_T$ of a tree $T$.
Note that we can assign directions to every edge of $T$ and end up with an \Index{oriented tree} $D_T$. A directed edge from $v$ to $u$ will be called an \Index{arc} and can be represented by $\overrightarrow{vu}$.
The set of arcs will be represented by $\calA(D_T)$.
With respect to $D_T$, the \Index{in-neighborhood} of $v\in \calV(T)$ is
\[
N_{D_T}^-(v)\coloneqq \Set{u\in N_T(v)\mid \overrightarrow{uv}\in D}.
\]
One can similarly define the \Index{out-neighborhood} $N_{D_T}^+(v)$ of $v$ in $D_T$.  A vertex of $D_T$ is called a \Index{source} if it does not have in-neighbors. The oriented tree $D_T$ is called an \Index{out-tree} if it has exactly one source, which will be called the \Index{root} of $D_T$. An out-tree $D_T$ with root $v$ will be denoted by $D_{T,v}^+$. Obviously, given a vertex $v\in T$, one can build up a unique oriented tree $D_T$ with $D_T=D_{T,v}^+$: for every edge $\{u_1,u_2\}\in \calE(T)$, $\overrightarrow{u_1u_2}\in \calA(D_T)$ if and only if $\dis_T(v,u_1)<\dis_T(v,u_2)$.

Given an oriented tree $D_T$, the subset
\[
A_{D_T}\coloneqq \Set{x_{u,v}\mid \overrightarrow{uv}\in\calA(D_T)} \subset \calV(G_T)
\]
will be called an \Index{orientation assignment} (associated to $D_T$).
Obviously, $|A_{D_T}|=|\calV(T)|-1$  and $|A_{D_T}\cap \{x_{u,v},x_{v,u}\}|=1$ for every edge $\{u,v\}\in T$.
Conversely, given a subset $A\subset G_T$ such that $|A|=|\calV(T)|-1$  and $|A\cap \{x_{u,v},x_{v,u}\}|=1$ for every edge $\{u,v\}\in T$, one can recover easily the oriented tree $D_T$ such that $A=A_{D_T}$. If $D_T$ is an out-tree, we will also call $A_{D_T}$ an \Index{out-tree orientation assignment}.

\begin{proposition}
    \label{facet of independence complex of G_T}
    Let $\calI(G_T)$ be the independence complex of $G_T$, where $G_T$ is the generic graph of a tree $T$. Then the facet set
    \[
    \calF(\calI(G_T))= \Set{A \subset G_T \mid \text{$A$ is an out-tree orientation assignment}}.
    \]
    In particular, $\calI(G_T)$ is pure of dimension $|\calV(T)|-2$.
\end{proposition}

\begin{proof}
    We may assume that $|\calV(T)|=n$. Since $|\calV(G_T)|=2(n-1)$, for every subset $B\subset \calV(G_T)$ with $|B|\ge n$, one can find suitable $x_{i,j},x_{j,i}\in B$ by the pigeonhole principle. But these two vertices are adjacent in $G_T$, meaning that $B$ is dependent. Hence the cardinality of any independent set of $G_T$ is at most $n-1$.

    On the other hand, if $B$ is an independent set containing less than $n-1$ vertices, we claim that there exists some vertex $x_{i,j} \in \calV(G_T)$, such that $B \cup \{x_{i,j}\}$ is an independent set. Hence every maximal independent set contains exactly $n-1$ vertices.
    In fact, since $|B| < n-1$, there exists some $x_{i,j} \in \calV(G_T)$, such that $x_{i,j} \notin B$ and $x_{j,i} \notin B$. If $x_{j,i}$ is not adjacent to any vertex in $B$, then clearly $B \cup \{x_{j,i}\}$ is an independent set. Otherwise, assume that there exists some $x_{l,k} \in B$, such that $\{x_{j,i}, x_{l,k}\} \in \calE(G_T)$. It is easy to see that neighborhoods satisfy $N(x_{i,j}) \subseteq N(x_{l,k})$. Since $x_{l,k}$ is a vertex in the independent set $B$,
    $N(x_{l,k})\cap B=\varnothing$. Therefore $N(x_{i,j})\cap B=\varnothing$, i.e., $B \cup \{x_{i,j}\}$ is also an independent set.

    To establish the description of the facet set, we first show that every out-tree orientation assignment of $G_T$ is a facet of $\calI(G_T)$. Let $A=A_{D_T}$ be such an assignment, and let $v$ be the root of $D_T$. Then for every distinct vertices $x_{u_1,v_1}$ and $x_{u_2,v_2}$ in $A$, $\dis_T(v,u_1)<\dis_T(v,v_1)$ and $\dis_T(v,u_2)<\dis_T(v,v_2)$. We claim that $x_{u_1,v_1}$ and $x_{u_2,v_2}$ are not adjacent in $G_T$. In fact, if $\{x_{u_1,v_1}, x_{u_2,v_2}\} \in \calE(G_T)$, then for the root $v$, the farthest vertex among $u_1, v_1, u_2, v_2$ is either $u_1$ or $u_2$, a contradiction.  On the other hand, for any vertex $x_{i,j} \in \calV(G_T) \setminus A$, $x_{j,i} \in A$ and these two vertices are adjacent in $G_T$. Therefore, $A$ cannot be properly expanded to a bigger independent set.

    Finally, we will show that every independent set $A$ with $n-1$ vertices in $G_T$ is an out-tree orientation assignment.
    Note that any such set $A$ is an orientation assignment. Say $A=A_{D_T}$. It follows from \cite[Proposition 2.1.1]{MR2472389}
    that $D_T$ contains at least one source.  Let $v$ be one such vertex. If $D_T\ne D_{T,v}^+$, then there exists $x_{u,u'} \in A$, such that $\dis_T(v, u) > \dis_T(v, u')$.
    Let $w=e(u,v)$. As $v$ is a source,
    $x_{v,w} \in A$ and  $\{x_{v,w}, x_{u,u'}\} \in \calE(G_T)$. This contradicts to the assumption that $A$ is an independent set. Hence $A=A_{D_{T,v}^+}$ is an out-tree orientation assignment. This completes the proof.
\end{proof}

\begin{lemma}
    Let $A_{D_{T,u}^+}$ and $A_{D_{T,v}^+}$ be two distinct out-tree orientation assignments. Then $\dis_{\calI(G_T)}(A_{D_{T,u}^+}, A_{D_{T,v}^+})=\dis_T(u,v)$.
    \label{lem:Assign-Tree}
\end{lemma}

\begin{proof}
    When $\{u,v\}$ is an edge of $T$, it is clear that $A_{D_{T,u}^+}$ and $A_{D_{T,v}^+}$ exchange $x_{u,v}$ with $x_{v,u}$. Hence $\dis(A_{D_{T,u}^+}, A_{D_{T,v}^+})=1$. As the distance function on $\calF(\calI(G_T))$ satisfies the triangle inequality, for general $u,v\in \calV(T)$, one has $\dis_{\calI(G_T)}(A_{D_{T,u}^+}, A_{D_{T,v}^+})\le \dis_T(u,v)$.

    On the other hand, suppose that $\dis_T(u,v)=t\ge 2$. Say $u=x_{i_0},x_{i_1},\dots,x_{i_t}=v$ is the unique path from $u$ to $v$ in $T$. Then
    \[
    \{x_{i_0,i_1},x_{i_1,i_2},\dots,x_{i_{t-1},i_t}\} \subseteq A_{D_{T,u}^+}\setminus A_{D_{T,v}^+}.
    \]
    Hence,
    \[
    \dis(A_{D_{T,u}^+}, A_{D_{T,v}^+})=|A_{D_{T,u}^+}\setminus A_{D_{T,v}^+}|\ge t=\dis_T(u,v). \qedhere
    \]
\end{proof}

\begin{remark}
    Let $G$ be a finite simple graph. One can check directly that the Stanley-Reisner complex $\Delta_G$ of the edge ideal $I(G)$ is actually the independence complex $\calI(G)$. Its Alexander dual $\calI(G)^\vee$ is $\braket{E^c \mid E\in\calE(G)}$, where the complement is taken with respect to $\calV(G)$. Notice that $\calI(G)^\vee$ is strongly shellable if and only if $G$ is ESS by Lemma \ref{complement shellable}.
\end{remark}

\begin{theorem}
    \label{bi-strongly shellable theorem}
    For any tree $T$,  the generic graph $G_T$ is bi-SS.
\end{theorem}

\begin{proof}
    By Proposition \ref{tree-graph}, we have already seen that $G_T$ is ESS. Therefore, it remains to show that the independence complex of $G_T$ is strongly shellable. By Proposition \ref{facet of independence complex of G_T}, it suffices to investigate the out-tree orientation assignments of $G_T$.

    There exists a total order on $\calV(T)$: $u_1, \dots, u_n$, such that the induced subgraph $T_k$ of $T$ on $\{u_1, \dots, u_k\}$ is connected for each $k \in [n]$. We will simply write $A_{D_{T,u_i}^+}$ as $A_{u_i}$.  It remains to prove that $A_{u_1}, \dots, A_{u_n}$ is a strong shelling order on $\calF(\calI(G_T))$. In fact, for each pair $i<j$, since $T_j$ is connected, there exists a $k<j$, such that $\{u_k, u_j\} \in \calE(T)$ and $\dis_T(u_i, u_j)=\dis_T(u_i, u_k)+1$. By applying Lemma \ref{lem:Assign-Tree} and Lemma \ref{d=d-1+1}, one can complete the proof.
\end{proof}

As for pure simplicial complexes, strong shellability implies Cohen-Macaulayness, we see immediately that Theorem \ref{bi-strongly shellable theorem} generalizes Proposition \ref{bi-CM}.

\begin{remark}
    We considered the codimension one graph $\Gamma(\Delta)$ of a pure simplicial complex $\Delta$ in \cite{SSC}.
    This is a finite simple graph whose vertex set is $\mathcal{F}(\Delta)$, and two vertices $F, \, G$ are adjacent in $\Gamma(\Delta)$ if and only if $|F\setminus G|=1$.

    Now given the generic graph $G_T$ of some tree $T$, we look that $\Gamma(\calI(G_T))$. It follows immediately from Lemma \ref{lem:Assign-Tree} that $\Gamma(\calI(G_T))$ is isomorphic to $T$.  This fact means that from the generic graph $G_T$ of some tree $T$, we can indeed recover $T$ up to isomorphism.

    Also from this point of view, the strong shellability of $\calI(G_T)$ in Theorem \ref{bi-strongly shellable theorem} can be regarded as a special case of \cite[Theorem 4.7]{SSC}, where we gave a characterization of strongly shellable pure complexes in terms of their codimension one graphs.
\end{remark}

\section{ESS bipartite graphs}
Recall that a \Index{Ferrers graph} is a bipartite graph $G$ on two disjoint vertex set $X=\{x_1,\dots,x_m\}$ and $Y=\{y_1,\dots,y_n\}$ such that if $\{x_i,y_j\}\in \calE(G)$, then so is $\{x_r,y_s\}$ for $1\le r \le i$ and $1\le s \le j$.  It follows from \cite[Theorems 4.1, 4.2]{ MR2457403} that a bipartite graph $G$ with no isolated vertex is a Ferrers graph if and only if the complement graph $\bar{G}$ is chordal. Therefore, by Theorem \ref{ssg and chordal graph}, a bipartite graph  with no isolated vertex is ESS if and only if it is a Ferrers graph. In this section, we will study the ESS property of these graphs from a new point of view.

Let $G$ be a finite simple graph. Take arbitrary vertex $w$ from $G$.
The set of vertices which has distance $i$ from $w$ will be denoted by $D(w, i)$. Clearly, $D(w, 0)=\{w\}$ and $D(w, 1)=N(w)$, the neighborhood of $w$.

\begin{definition}
    For a vertex $x \in D(w, i)$, the \Index{upward neighborhood} of $x$ with respect to $w$, denoted by $\uN_w(x)$, is the set $N(x) \cap D(w, i+1)$. The cardinality of this set will be called the \Index{upward degree} of $x$ with respect to $w$ and denoted by $\ud_w(x)$. The \Index{upward degree sequence} on $D(w, i)$ is a sequence of upward degrees of all vertices in $D(w, i)$ arranged in decreasing order:
    \[
    \ud_w(x_1) \geq  \ud_w(x_2) \geq \cdots \geq \ud_w(x_t).
    \]
    Correspondingly, we can define the \Index{downward neighborhood}  $\dN_w(x)$. 
\end{definition}

Let $G$ be a simple graph. Recall that the \Index{line graph} $L(G)$ of $G$ is the graph such that $\calV(L(G))=\calE(G)$ and $\Set{e_1,e_2}\in \calE(L(G))$ if and only if $e_1$ and $e_2$ are adjacent in $G$.  For simplicity, if two edges $e_1$ and $e_2$ of $G$ satisfies $\dis_{L(G)}(e_1,e_2)=d$, we will also say that they have distance $d$ in $G$.

In particular, if $G$ is an ESS graph,
then every two non-adjacent edges of $G$ will be adjacent to a common edge in $G$. In other words, every two non-adjacent edges of an ESS graph have distance $2$.
Consequently, the diameter of $G$ is at most $3$. Hence, for any vertex $w$ in $G$, we have the partition of the vertex set
\[
\calV(G)=D(w, 0) \sqcup D(w, 1) \sqcup D(w, 2) \sqcup D(w, 3). 
\]

\begin{lemma}
    \label{inclusion}
    Let $w$ be a vertex of an ESS bipartite graph $G$. For $i=1$ or $2$ and two vertices $x, y \in D(w, i)$, we have either $\uN_w(x) \subseteq \uN_w(y)$ or $\uN_w(y) \subseteq\uN_w(x)$.
\end{lemma}

\begin{proof}
    Assume for contradiction that we can find $x' \in \uN_w(x) \setminus \uN_w(y)$ and $y' \in \uN_w(y) \setminus \uN_w(x)$. Then $\{x, x'\}, \, \{y, y'\} \in \calE(G)$, and $\{x, y'\}, \, \{y, x'\} \notin \calE(G)$. Since $G$ is a bipartite graph, we have $\{x, y\}, \, \{x', y'\} \notin \calE(G)$. Thus, the distance between $\{x, x'\}$ and $\{y, y'\}$ is more than 2. This is a contradiction since every two non-adjacent edges of an ESS graph have distance $2$.
\end{proof}

\begin{corollary}
    \label{decreasing}
    Let $w$ be a vertex of an ESS bipartite graph $G$. For $i=1$ or $2$, if $|D(w,i)|=t$, we can order the vertices in $D(w,i): x_1,x_2,\dots,x_t$ such that $\uN_w(x_1)\supseteq  \cdots \supseteq \uN_w(x_t)$. In particular, $\ud_w(x_1)\ge \cdots \ge \ud_w(x_t)$.
\end{corollary}

\begin{lemma}
    \label{D(w, 2)}
    Let $w$ be a vertex of an ESS bipartite graph $G$. For any vertex $x \in D(w, 2)$ with $\uN_w(x) \neq \varnothing$, we have $\dN_w(x)=D(w, 1)$.
\end{lemma}

\begin{proof}
    Since $\uN_w(x) \neq \varnothing$, we have $y \in D(w, 3)$ such that $\{x, y\}$ is an edge of $G$. Assume that there exists $z \in D(w, 1) \setminus \dN_w(x)$. We claim that the distance between $\{x, y\}$ and $\{w, z\}$ is more than 2. In fact, the claim follows from the facts that $\{x, z\} \notin \calE(G)$, $\dis_G(w, x)=2$ and $\dis_G(w, y)=3$.  But this is again a contradiction since every two non-adjacent edges of an ESS graph have distance $2$.
\end{proof}

\begin{corollary}
    Let $w$ be a vertex of an ESS bipartite graph $G$.
    We order the vertices in $D(w,1): x_1,\dots,x_t$ and in $D(w,2):y_1,\dots,y_s$ as in Corollary \ref{decreasing}. If $\ud_w(y_i) \ge 1$ for some $i\in[s]$, then $y_i\in \uN_w(x_t)$. In particular, this index $i$ is at most  $\ud_w(x_t)$.
\end{corollary}

Now, we are ready to explain how to re-construct an ESS bipartite graph $G$ which has no isolated vertex.

\begin{construction}
    \label{ESS bipartite graph}
    Given two upward degree sequences $d_1\ge\cdots\ge d_t$ and $d_1'\ge \cdots \ge d_{d_t}'\ge d_{d_t+1}'=\cdots=d_{d_1}'=0$, we construct a graph $G$ as follows:
    \begin{enumerate}[1]
        \item Start with a vertex, which will be called $w$.
        \item Choose a set of new vertices $V(w, 1)=\{x_1, \dots, x_t\}$ for $G$. We require that  $w$ is adjacent to all vertices in $V(w, 1)$ in $G$. For each $i\in[t]$, label the vertex $x_i$ implicitly by the weight $d_i$.
        \item Choose a set of new vertices $V(w, 2)=\{y_1, \dots, y_{d_1}\}$ for $G$. For each $i\in [t]$, we require that $x_i$ is adjacent to all the initial $d_i$ vertices in $V(w, 2)$ in $G$. For each $j\in [d_1]$, label the vertex $y_j$ implicitly by the weight $d_j'$.
        \item Choose the final set of new vertices $V(w, 3)=\{z_1, \dots, z_{d'_1}\}$ for $G$. For each $k\in[d_t]$, we require that $y_k$ is adjacent to all the initial $d_k'$ vertices in $V(w,3)$ in $G$.
    \end{enumerate}
    The graph $G$ here will be called a graph \Index{constructed (from vertex $w$) by upward degree sequences}.
\end{construction}

\begin{example}
    In Figure \ref{Construction by upward degree sequences}, we have a graph constructed by upward degree sequences $4\ge 3 \ge 2$ and $2\ge 1 \ge 0 = 0$.
    \begin{figure}[!ht]
        \begin{minipage}[h]{\linewidth} \centering
            \includegraphics{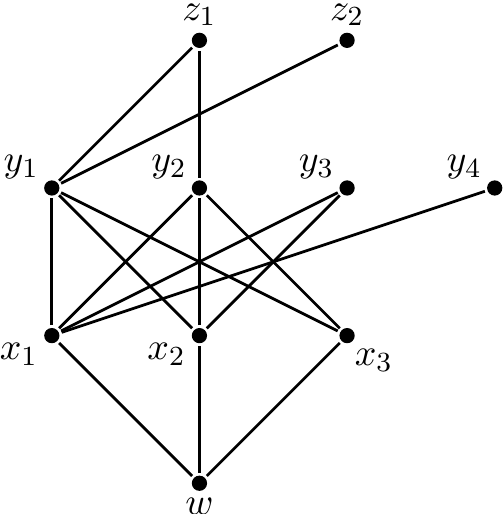}
        \end{minipage}
        \caption{Construction by upward degree sequences}
        \label{Construction by upward degree sequences}
    \end{figure}
\end{example}

\begin{remark}
    \label{explain ssbg}
    \begin{enumerate}[1]
        \item Let $G$ be a graph constructed by upward sequences as above. It is clear that $G$ is bipartite and $\dis_G(w,x)\le 3$ for any $x\in G$. Furthermore, $D(w, i) = V(w, i)$ for $i=1, 2, 3$.
        \item In addition, if $d_t=0$, then there exists no vertex in $D(w, 2)$ which has nonempty upward neighborhood. Hence $V(w, 3) = \varnothing$.
    \end{enumerate}
\end{remark}

\begin{theorem}
    \label{ssbg=cbuds}
    Let $G$ be a connected finite simple graph. Then the following conditions are equivalent:
    \begin{enumerate}[1]
        \item \label{ssbg=cbuds-1} $G$ is an ESS bipartite graph.
        \item \label{ssbg=cbuds-2} $G$ can be constructed from any vertex of $G$ by upward degree sequences.
        \item \label{ssbg=cbuds-3} $G$ can be constructed from a vertex of $G$ by upward degree sequences.
    \end{enumerate}
\end{theorem}

\begin{proof}
    \ref{ssbg=cbuds-1}$\Rightarrow$\ref{ssbg=cbuds-2}: Let $G$ be an ESS bipartite graph. Choose a vertex $w$  arbitrarily and set $V(w, 1)=D(w, 1)$. It is clear that $w$ is adjacent to any vertex in $V(w, 1)$. By Lemma \ref{inclusion}, there exists a total order $x_1, \dots, x_t$
    of the vertices in $V(w,1)$ such that
    \begin{equation}
        \label{uN-1} \tag{\dag}
        \uN_w(x_1)\supseteq \cdots \supseteq \uN_w(x_t).
    \end{equation}
    If we denote the upward degree of $x_i$ with respect to $w$ by $d_i$, then obviously $d_1 \geq d_2 \geq \cdots \geq d_t$.

    Following from the containment in \eqref{uN-1}, there exists a total order $y_1,\dots,y_{d_1}$ of vertices in $\uN_w(x_1)$, such that $y_j\in \uN_w(x_i)$ whenever $j\le d_i$. Note that if $y\in D(w,2)$, then $y$ is adjacent to some $x_i\in D(w,1)$. Hence $y\in \uN_w(x_i)\subseteq \uN_w(x_1)$. Therefore $D(w,2)=\uN_w(x_1)$. We will set $V(w, 2)=D(w, 2)$.

    By Lemma \ref{D(w, 2)}, any $y \in D(w,2)$ with $\uN_w(y) \neq \varnothing$ will satisfy $y \in \uN_w(x_t)$.
    Note that $\uN_w(x_t)=\Set{y_1,\dots,y_{d_t}}$. By Lemma \ref{inclusion}, we may assume that
    \begin{equation}
        \label{uN-2}  \tag{\ddag}
        \uN_w(y_1)\supseteq \cdots \supseteq \uN_w(y_{d_t}).
    \end{equation}
    If we denote the upward degree of $y_i$ with respect to $w$ by $d_i'$, then obviously
    $d_1'\ge \cdots \ge d_{d_t}'\ge d_{d_t+1}'=\cdots=d_{d_1}'=0$.

    Following from the containment in \eqref{uN-2}, there exists a total order $z_1,\dots,z_{d_1'}$ of vertices in $\uN_w(y_1)$, such that $z_j\in \uN_w(y_i)$ whenever $j\le d_i'$. Similar to the previous argument, we will have $D(w,3)=\uN_w(y_1)$.  Set $V(w, 3)=\uN_w(y_1)$.

    Since the diameter of $G$ is at most $3$, we have
    \begin{align*}
        V(G)=& \{w\} \sqcup D(w, 1) \sqcup D(w, 2) \sqcup D(w, 3) \\
        = & \{w\} \sqcup V(w, 1) \sqcup V(w, 2) \sqcup V(w, 3).
    \end{align*}
    By the above discussion, $G$ is constructed from $w$ by upward degree sequences.

    \ref{ssbg=cbuds-2}$\Rightarrow$\ref{ssbg=cbuds-3} is clear.

    \ref{ssbg=cbuds-3}$\Rightarrow$\ref{ssbg=cbuds-1}: Let $G$ be the graph constructed by Construction \ref{ESS bipartite graph}. It is easy to see that $G$ is a bipartite graph with no isolated vertex. In the following, we will show that $G$ is ESS. Let's assign an total order $\succ$ on $V(G)$:
    \[
    w \succ x_1 \succ \cdots \succ x_t \succ y_1 \succ \cdots \succ y_{d_1} \succ z_1 \succ \cdots \succ  z_{d'_1}.
    \]
    For simplicity, for each edge $\{a_1, a_2\}$ in $G$, we always assume that $a_1 \succ a_2$. We will consider the lexicographic order $\succ_{lex}$ on the edge set of $G$ with respect to $\succ$, i.e., for every pair of edges $\{a_1, a_2\}, \{b_1, b_2\} \in \calE(G)$,
    \[
    \text{$\{a_1, a_2\} \succ_{lex} \{b_1, b_2\}$ if and only if $a_1 \succ b_1$, or $a_1 = b_1$ and $a_2 \succ b_2$.}
    \]
    We claim that the lexicographic order $\succ_{lex}$ on $\calE(G)$ is a strong shelling order.
    Take two distinct edges $\{a_1,a_2\}\succ_{lex}\{b_1,b_2\}$. We may assume that these two edges are disjoint. Thus, it suffices to consider the following cases:
    \begin{enumerate}[i]
        \item Suppose that $a_1$ and $b_1$ belong to the same set $V(w, i)$ for some $i$. As $a_1\succ b_1$, $\uN_w(a_1)\supseteq \uN_w(b_1)$. Thus, we have the connecting edge $\{a_1, b_2\} \in \calE(G)$. It is obvious that $\{a_1,b_2\}\succ_{lex} \{b_1,b_2\}$.
        \item Suppose that $a_1=w$ and $b_1 \in V(w,1)$. Then we have the connecting edge $\{a_1, b_1\} \in \calE(G)$. It is obvious that $\{a_1,b_1\}\succ_{lex} \{b_1,b_2\}$.
        \item Suppose that $a_1=w$ and $b_1 \in V(w,2)$.  As $b_2\in \uN_w(b_1)\ne \varnothing$, $\dN_w(b_1)=V(w,1)$ by Lemma \ref{D(w, 2)}.  Hence we have the connecting edge $\{a_2, b_1\} \in \calE(G)$.  It is obvious that $\{a_2,b_1\}\succ_{lex} \{b_1,b_2\}$.
        \item Suppose that $a_1 \in V(w,1)$ and $b_1 \in V(w,2)$. As the previous case, $\dN_w(b_1)=V(w,1)$.  Hence we have the connecting edge $\{a_1, b_1\} \in \calE(G)$.  It is obvious that $\{a_1,b_1\}\succ_{lex} \{b_1,b_2\}$.
    \end{enumerate}
    Thus $\succ_{lex}$ is a strong shelling order on $\calE(G)$.
\end{proof}

\begin{corollary}
    \label{distance 2}
    Let $G$ be an ESS bipartite graph which has no isolated vertex. Then there exists a vertex $w'$, such that the distance between $w'$ and any vertex of $G$ is at most 2.
\end{corollary}

\begin{proof}
    As a consequence of Theorem \ref{ssbg=cbuds}, we will assume that $G$ is constructed from the vertex $w$ by the associated upward degree sequences, as in Construction \ref{ESS bipartite graph}. Consider the following cases:
    \begin{enumerate}[1]
        \item If $d_t=0$, by Remark \ref{explain ssbg}, we will have $V(w, 3)=\varnothing$. Hence there exists no vertex in $G$ which has distance 3 from $w$.
        \item If $d_t \neq 0$, we claim that $w'=y_1$ will be sufficient. Apparently, $\dis_G(w, y_1)=2$ and $y_1$ is adjacent to any vertex in $V(w, 1) \cup V(w, 3)$. For any distinct vertex $y_i \in V(w, 2)$, it is clear that $\{x_1,y_i\}, \{x_1, y_1\} \in \calE(G)$. Hence $\dis_G(y_i, y_1)=2$ for each $2 \leq i \leq d_1$. \qedhere
    \end{enumerate}
\end{proof}

\begin{remark}
    By Corollary \ref{distance 2}, if $G$ is an ESS bipartite graph with no isolated vertex, we can choose a proper vertex $w$, such that $V(G) = \{w\} \sqcup D(w, 1) \sqcup D(w, 2)$. Assume that $D(w, 1) = \{x_1, x_2, \dots, x_t\}$ and $D(w, 2) = \{y_1, y_2, \dots, y_m\}$ as in Construction \ref{ESS bipartite graph}. If we set $y_0=w$, then the two disjoint parts  $\{x_1, x_2, \dots, x_t\}$ and $\{y_0, y_1, \dots, y_m\}$ satisfy:
    \[
    \text{if $\{x_i,y_j\} \in \calE(G)$, then $\{x_t,y_s\} \in \calE(G)$ for all $t \leq i$ and $s \leq j$.}
    \]
    This shows that $G$ is a Ferrers graph.
\end{remark}


\begin{bibdiv}
\begin{biblist}

\bib{MR1251956}{book}{
      author={Bruns, Winfried},
      author={Herzog, J{\"u}rgen},
       title={Cohen-{M}acaulay rings},
      series={Cambridge Studies in Advanced Mathematics},
   publisher={Cambridge University Press},
     address={Cambridge},
        date={1993},
      volume={39},
        ISBN={0-521-41068-1},
}

\bib{MR1341789}{article}{
      author={Bruns, Winfried},
      author={Herzog, J{\"u}rgen},
       title={On multigraded resolutions},
        date={1995},
        ISSN={0305-0041},
     journal={Math. Proc. Cambridge Philos. Soc.},
      volume={118},
       pages={245\ndash 257},
}

\bib{MR2472389}{book}{
      author={Bang-Jensen, J{\o}rgen},
      author={Gutin, Gregory},
       title={Digraphs: theory, algorithms and applications.},
     edition={Second},
      series={Springer Monographs in Mathematics},
   publisher={Springer-Verlag London Ltd.},
     address={London},
        date={2009},
        ISBN={978-1-84800-997-4},
}

\bib{arXiv:1508.03799}{article}{
      author={Bigdeli, Mina},
      author={Yazdan~Pour, Ali~Akbar},
      author={Zaare-Nahandi, Rashid},
       title={Stability of betti numbers under reduction processes: towards
  chordality of clutters},
        date={2015},
      eprint={arXiv:1508.03799},
}

\bib{MR2457403}{article}{
      author={Corso, Alberto},
      author={Nagel, Uwe},
       title={Monomial and toric ideals associated to {F}errers graphs},
        date={2009},
        ISSN={0002-9947},
     journal={Trans. Amer. Math. Soc.},
      volume={361},
       pages={1371\ndash 1395},
         url={http://dx.doi.org/10.1090/S0002-9947-08-04636-9},
}

\bib{MR0130190}{article}{
      author={Dirac, Gabriel~Andrew},
       title={On rigid circuit graphs},
        date={1961},
        ISSN={0025-5858},
     journal={Abh. Math. Sem. Univ. Hamburg},
      volume={25},
       pages={71\ndash 76},
}

\bib{MR2853077}{article}{
      author={Emtander, Eric},
      author={Mohammadi, Fatemeh},
      author={Moradi, Somayeh},
       title={Some algebraic properties of hypergraphs},
        date={2011},
        ISSN={0011-4642},
     journal={Czechoslovak Math. J.},
      volume={61(136)},
       pages={577\ndash 607},
         url={http://dx.doi.org/10.1007/s10587-011-0031-0},
}

\bib{MR2603461}{article}{
      author={Emtander, Eric},
       title={A class of hypergraphs that generalizes chordal graphs},
        date={2010},
        ISSN={0025-5521},
     journal={Math. Scand.},
      volume={106},
       pages={50\ndash 66},
}

\bib{MR1633767}{article}{
      author={Eagon, John~A.},
      author={Reiner, Victor},
       title={Resolutions of {S}tanley-{R}eisner rings and {A}lexander
  duality},
        date={1998},
        ISSN={0022-4049},
     journal={J. Pure Appl. Algebra},
      volume={130},
       pages={265\ndash 275},
}

\bib{MR0186421}{article}{
      author={Fulkerson, D.~R.},
      author={Gross, O.~A.},
       title={Incidence matrices and interval graphs},
        date={1965},
        ISSN={0030-8730},
     journal={Pacific J. Math.},
      volume={15},
       pages={835\ndash 855},
}

\bib{MR1171260}{incollection}{
      author={Fr{\"o}berg, Ralf},
       title={On {S}tanley-{R}eisner rings},
        date={1990},
   booktitle={Topics in algebra, {P}art 2 ({W}arsaw, 1988)},
      series={Banach Center Publ.},
      volume={26},
   publisher={PWN},
     address={Warsaw},
       pages={57\ndash 70},
}

\bib{SSC}{unpublished}{
      author={Guo, Jin},
      author={Shen, Yi-Huang},
      author={Wu, Tongsuo},
       title={Strong shellability of pure simplicial complexes},
        note={preprint},
}

\bib{MR2724673}{book}{
      author={Herzog, J{\"u}rgen},
      author={Hibi, Takayuki},
       title={Monomial ideals},
      series={Graduate Texts in Mathematics},
   publisher={Springer-Verlag London Ltd.},
     address={London},
        date={2011},
      volume={260},
        ISBN={978-0-85729-105-9},
}

\bib{arXiv:1508.07119}{article}{
      author={{Herzog}, J\"urgen},
      author={{Rahimi}, Ahad},
       title={{Bi-Cohen-Macaulay graphs.}},
    language={English},
        date={2016},
        ISSN={1077-8926/e},
     journal={{Electron. J. Comb.}},
      volume={23},
       pages={research paper p1.1, 14},
}

\bib{MR1918513}{article}{
      author={Herzog, J{\"u}rgen},
      author={Takayama, Yukihide},
       title={Resolutions by mapping cones},
        date={2002},
        ISSN={1512-0139},
     journal={Homology Homotopy Appl.},
      volume={4},
       pages={277\ndash 294},
        note={The Roos Festschrift volume, 2},
}

\bib{arXiv:1601.00456}{article}{
      author={Moradi, Somayeh},
      author={Khosh-Ahang, Fahimeh},
       title={Expansion of a simplicial complex},
        date={2016},
        ISSN={0219-4988},
     journal={J. Algebra Appl.},
      volume={15},
       pages={1650004},
         url={http://dx.doi.org/10.1142/S0219498816500043},
}

\bib{MR2457194}{article}{
      author={Naeem, Muhammad},
       title={Cohen-{M}acaulay monomial ideals of codimension 2},
        date={2008},
        ISSN={0025-2611},
     journal={Manuscripta Math.},
      volume={127},
       pages={533\ndash 545},
         url={http://dx.doi.org/10.1007/s00229-008-0217-4},
}

\bib{arXiv:1601.03207}{article}{
      author={Nikseresht, Ashkan},
      author={Zaare-Nahandi, Rashid},
       title={On generalizations of cycles and chordality to hypergraphs},
      eprint={arXiv:1601.03207},
}

\bib{arXiv:1511.04676}{article}{
      author={Rahmati-Asghar, Rahim},
      author={Moradi, Somayeh},
       title={On the stanley--reisner ideal of an expanded simplicial complex},
        date={2016},
        ISSN={1432-1785},
     journal={Manuscripta Mathematica},
       pages={1\ndash 13},
         url={http://dx.doi.org/10.1007/s00229-016-0822-6},
}

\bib{MR0270957}{article}{
      author={Rose, Donald~J.},
       title={Triangulated graphs and the elimination process},
        date={1970},
        ISSN={0022-247x},
     journal={J. Math. Anal. Appl.},
      volume={32},
       pages={597\ndash 609},
}

\bib{MR1453579}{book}{
      author={Stanley, Richard~P.},
       title={Combinatorics and commutative algebra},
     edition={Second},
      series={Progress in Mathematics},
   publisher={Birkh\"auser Boston Inc.},
     address={Boston, MA},
        date={1996},
      volume={41},
        ISBN={0-8176-3836-9},
}

\bib{MR2853065}{article}{
      author={Woodroofe, Russ},
       title={Chordal and sequentially {C}ohen-{M}acaulay clutters},
        date={2011},
        ISSN={1077-8926},
     journal={Electron. J. Combin.},
      volume={18},
       pages={Paper 208, 20},
}

\end{biblist}
\end{bibdiv}

\end{document}